\documentclass{amsart}

\pdfoutput=1

\usepackage{amsthm} 
\usepackage{color}

\usepackage{pdfpages}

\usepackage{graphicx,epsfig}
\usepackage{epstopdf}
\usepackage{amsmath, amssymb, latexsym, euscript,amscd}
\usepackage{url}
\usepackage[all]{xy}
\usepackage{psfrag}
\usepackage{mathrsfs}
\usepackage{tikz}
\usetikzlibrary{shapes.geometric}
\usepackage{subcaption}
\usepackage{multicol}
\usepackage{paracol}

\usetikzlibrary{shapes.geometric}
\usetikzlibrary{positioning}

\newcounter{notes}%

\definecolor{darkgreen}{rgb}{0.0, 0.5, 0.0}

\newtheorem{theorem}{Theorem}[section]
\newtheorem{lemma}[theorem]{Lemma}
 
\newtheorem{definition}[theorem]{Definition} 
\newtheorem{proposition}[theorem]{Proposition}
\newtheorem{remark}[theorem]{Remark} 
\newtheorem{question}[theorem]{Question}

\def\smallskip{\vspace{.15cm}}
\def\medskip{\vspace{.3cm}}
\def\text{\mbox}
\def\rh2{{\mathbb R}{\mathbb H}^2}
\def\ch2{{\mathbb C}{\mathbb H}^2}

\def\ZZ{{\mathbb Z}}
\def\NN{{\mathbb N}}

\def\RP2{{\mathbb{RP}}^2}
\def\RP3{{\mathbb{RP}}^3}
\def\RP{{\mathbb{RP}}}

\def\int{\operatorname{int}}

\def\H2R{{\mathbb H}^2\times {\mathbb R}}
\def\eps{\varepsilon}

\def\C2{\operatorname{C^2}}

\def\lcm{\operatorname{lcm}}

\definecolor{back}{RGB}{255,255,255}
\definecolor{fore}{RGB}{0,0,0}
\definecolor{title}{RGB}{255,0,90}

\definecolor{green}{rgb}{0.0, 0.5, 0.0}
\definecolor{purple}{rgb}{0.5, 0.0, 0.5}
\definecolor{bluegreen}{rgb}{0.0,0.5, 0.5}
\definecolor{orange}{rgb}{1,0.5, 0.1}
\definecolor{redgreen}{rgb}{0.5, 0.5, 0.0}

\def\green{\color{green}}

\def\green{\color{green}}

\def\g2{{\green 2}}

\newcommand{\bv}{\left[\begin{array}{c}}
\newcommand{\ev}{\end{array}\right]}
\newcommand{\bbmat}{\begin{bmatrix}} 
\newcommand{\ebmat}{\end{bmatrix}}
\newcommand{\bmat}{\begin{matrix}} 
\newcommand{\emat}{\end{matrix}}
\newcommand{\bpmat}{\begin{pmatrix}} 
\newcommand{\epmat}{\end{pmatrix}}
\newcommand{\ra}{\rightarrow}
\newcommand{\mP}{\mathbb P}

\begin{document}
\title{almost-regular Dessins d'Enfant on a Torus and Sphere} 
\author{Joachim K\"{o}nig, Arielle Leitner, and Danny Neftin} 
\maketitle 

\begin{abstract} 
The Hurwitz problem asks which ramification data are {\em realizable}, i.e., appear as the ramification type of a covering. 
We use dessins d'enfant to show that families of genus $1$ regular ramification data with small changes are realizable with the exception of four families which were recently shown to be nonrealizable. A similar description holds in the case of genus $0$ ramification data. 
\end{abstract}

\section{Introduction}\label{sec:intro}
A central goal in algebra, topology and geometry is describing maps using simple combinatorial data. In the equivalent categories, of topological coverings, branched coverings   $f:R\ra S$ of (compact connected, orientable\footnote{The Hurwitz problem for nonorientable Riemann surfaces reduces to the orientable case \cite[Proposition 2.7]{EKS}.} ) Riemann surfaces, and algebraic morphisms (of smooth projective curves),  each map has a fundamental invariant called its {\it ramification type}. The latter is the multiset of ramification multisets $\{e_f(Q)\,|\,Q\in f^{-1}(P)\}$ where $P$ runs over branch points of $f$. 
For example, the ramification type of the map  $\mP^1\ra\mP^1$ given by $x^2$ is $[2][2]$, where $\mP^1=\mP^1_{\mathbb C}$ is the complex projective line (or equivalently the $2$-dimensional real sphere). 

The Hurwitz (existence) problem is a classical question in the theory of coverings which seeks to classify under which conditions such combinatorial data arises from a covering. Since the problem remains open only in the case of coverings of $S=\mP^1$, we shall henceforth restrict to this case. An obvious necessary condition for such data is that it is consistent with the Riemann--Hurwitz  formula. 
Namely, a tuple $E$ of nontrivial partitions $E_1,\ldots,E_r$ of $n$ is called a {\it ramification data} if 
\begin{equation}
\label{equ:RH}
g_E:=1-n + \frac{1}{2}\cdot \sum_{i=1}^r\sum_{e\in E_i}(e-1)
\end{equation}
 is a nonnegative integer. The integers $n,g_E,r$ are called the {\it degree, genus, and number of branch points} of $E$, respectively. The Hurwitz problem then asks which ramification data are {\it realizable}  as the ramification type of a covering of $\mathbb{P}^1$? 
%

The case of ramification data of high genus is relatively well understood. Namely, if $E$ is a ramification data of degree $n$ and genus at least $(n+1)/2$ then $E$ is realizable, by Edmonds--Kulkarni--Stong \cite{EKS}. Moreover in view of recent computations, Zheng suggests \cite{Zheng} that every ramification type of genus at least $3$ and degree at least $5$ is realizable. In genus $2$, the only infinite family of nonrealizable ramification types found so far is $[2^*] [2^*] [2^*] [3,5,2^*]$, where $2^*$ denote that all remaining elements in the multiset equal $2$. 
Moreover, the nonoccuring genus $1$ ramification data in Zheng's \cite[Table 3]{Zheng}  are all ``almost-regular", that is, almost all entries in each multiset are equal. Such patterns can also be identified among the genus $0$ ramification data in \cite[Table 2]{Zheng}.

We consider families of almost-regular ramification data in the most subtle case where the ramification data is of genus is $0$ or $1$, that is, where the covering space is either  $\mP^1$ or a torus.  Here a family 
 $T = (T_i)_{i\in\mathbb N}$ of ramification data with $r$ branch points is called {\it almost-regular of type} $[k_1,\ldots,k_r]$ 
{\it and error\footnote{Note that there are many possible definitions of error, such as the one defined by taking $\eps_{i,j}:=\sum_{e\in T_{i,j}}|e-k_j|.$} at most $\eps$}, 
if there exist a tuple $A_j$ of positive integers different from $k_j$ for each $j=1,\ldots,r$, such that each $T_i$ is of the form $[A_1,k_1^{*}],\ldots,[A_r,k_r^{*}]$ 
with degree tending to infinity, and such that the sum  $\sum_{j=1}^r\sum_{a\in A_j}a\leq \eps$. 
For example, 
the ramification data $[2^*] [2^*] [2^*] [1^4 , 4, 2^*]$ are almost-Galois of type $[2,2,2,2] $ with error at most $8$. 

The ramification type of Galois coverings is {\it regular}, that is, almost-regular with error $0$. 
The ramification types of genus $0$ Galois coverings were already known to Klein \cite{Klein}, and the ramification types of genus $1$ Galois coverings are well known to be regular of type $[2,2,2,2]$, or $[3,3,3]$, or $[2,4,4]$, or $[2,3,6]$. 
Particular families\footnote{Many authors consider a generalized almost-regular ramification where one of the multisets has an absolutely bounded amount of entries, such as $[n/2,n/2] [2^*] [2^*]$, cf.~\cite{Thom}, \cite{Pakovitch}, \cite{PerPet}.} of almost-regular ramification data of the latter four types have been well studied. 
Notably, Pascali--Petronio \cite[Theorems 0.4, 0.5, 0.6]{PasPet} give necessary and sufficient arithmetic conditions on the degrees of the  genus $0$ ramification data  (1) $[1,2^*] [1,4^*] [1,4^*]$,  (2) $[1,3^*] [1,3^*] [1,3^*]$, and (3) $[1,2^*] [1,3^*] [1,6^*]$ to be realizable. Moreover, the realizability of many almost-regular families of the above types was also proved in the classification of monodromy groups of indecomposable coverings of low genus, see \cite{GS}, \cite{NZ}, \cite{NZ2}.

On the other hand, recently Do--Zieve \cite{Z} showed that the genus $1$ ramification data (A)  $[1,3,2^*]$ $[2^*]$ $[2^*]$ $[2^*]$ is nonrealizable, and conjectured the nonrealizability of the following data:
\begin{equation*}\label{equ:Zannier-types}
\, (B) \quad [2,4,3^*] [3^*] [3^*];\, (C) \quad [2^*] [3,5,4^*] [4^*]; \, (D) \quad [2^*] [3^*] [5,7,6^*]. 
\end{equation*} See Section \ref{nonex:genus1} for another proof of the nonrealizability of type (A). 
A simple proof of this conjecture was given by Corvaja--Zannier \cite{CZ}, and a proof can also be derived from Izmestiev--Kusner--Rote--Springborn--Sullivan\footnote{The results in \cite{IKRSS} imply that types (B)-(D) are nonrealizable but do not apply to (A). } \cite{IKRSS}. 

In this paper we give a method of extending realizations of small degree ramification data to realizations of families of almost-regular ramification data. These imply that with the exception of the above four families, all families of genus $1$ almost-regular ramification data with some bounded error are realizable, in contrast with the genus 0 situation discussed above. 

\begin{theorem} 
\label{thm:existence} Let $T$ be a family of almost-regular genus $1$ ramification data of type $[k_1,\ldots,k_r]$, error at most $\eps$, where $T_i$ is not one of the above exceptional types (A)-(D) for $i\in\mathbb N$. 
Then all but finitely many members of $T$ are realizable if $\eps\leq 6$, or if $[k_1,\ldots,k_r]\in\{[2,2,2,2] [3,3,3]\}$ and $\eps\leq 10$. 
\end{theorem}
%

Moreover, our method allows ``gluing" together realizations of almost-regular families of ramification data,  
into families with arbitrary amount of changes with bounded entries. For families of almost-regular ramification data of type $[2,2,2,2]$ this gives:
\begin{theorem}
\label{123exist}
Every family of almost-regular ramification data of genus $1$ of the form $[1^{k_1},3^{m_1},2^*]$ $[1^{k_2},3^{m_2},2^*]$ $[1^{k_3},3^{m_3},2^*]$ $[1^{k_4},3^{m_4},2^*]$ with fixed $k_i,m_i\in \NN\cup \{0\}$, $i=1,\ldots,4$, except for the type $[1,3,2^*]$ $[2^*] [2^*] [2^*]$ is realizable in all sufficiently large degrees. 
\end{theorem}


The Hurwitz problem can be viewed as combinatorial in nature, but the tools used in its study are diverse, coming from group theory, topology, and geometry. We use a mix of these techniques with a focus around {\it Dessins d'Enfant}. 
Given a ramification data $E$ of genus $g$ and three branch points,  a dessin for $E$ is a bicolored (nondirected) graph on a genus $g$ surface, where the multisets of valencies of white vertices,  black vertices, and faces, correspond to the three multisets of $E$, 
see Section \ref{sect:dessins} for the definition of the valency of a face, and the definition in the case of more than three branch points. The central example here is that of dessins corresponding to families of regular ramification data of types $[3,3,3], [2,4,4], [2,3,6]$ or $[2,2,2,2]$, which are regular tilings of the torus by hexagons or squares, called  {\it regular dessins}, cf.~Section \ref{loc_change}.  

It is well known  that the the realizability of a ramification data is equivalent to the existence of a dessin for that data. Thus by making an (absolutely) bounded number of changes to each member in a family of regular dessins $(C_i)_{i\in\mathbb N}$, such as adding or removing an edge or an isolated vertex, one obtains a family of dessins whose corresponding  ramification type is almost-regular. We call such families of dessins {\it almost-regular}. This transition from regular to almost-regular families of dessins is a special case of what will be referred to as {\emph{local changes}} later; see Definition \ref{def:close}.
To prove Theorem \ref{thm:existence}, we first realize a single ramification type $T_i$  as a dessin $C_i$  for $i \in \NN$. 
We then show that the changes made to obtain $C_i$ from a regular tiling, can also be made to a regular dessin with an arbitrary sufficiently large number of hexagons or squares, so that all $T_j$ of sufficiently large degree are realizable, and in fact by a family of almost-regular dessins. This yields a stronger version of Theorem \ref{thm:existence}, stated explicitly in Theorem \ref{thm:existence_strong}. 
To prove Theorem \ref{123exist} we show that the constructed almost-regular dessins can be glued together. 
%
This suggests that furthermore, the answer to the following question is positive: 
\begin{question} 
\label{conj:existence}
Is every family of almost-regular ramification data of genus $1$ different from types (A)-(D) above, with arbitrary error, 
realizable in all sufficiently large degrees\footnote{By the expression ``all degrees", we always mean ``all degrees which are compatible with the congruence conditions prescribed by the ramification type".}? 
Are such families realizable by  families of almost-regular dessins? 
\end{question}
Our proof of Theorem \ref{thm:existence} implies that the answer is positive for families with error $\eps\leq 6$ or families of type $[2,2,2,2]$ or $[3,3,3]$ with error $\eps\leq 10$. Theorem \ref{123exist} gives further evidence for a positive answer where $\eps$ is arbitrary but the entries are bounded. We note that our methods can be used to prove similar results for larger errors $\eps$ or larger uniform bounds on the entries but new methods are required to answer this question entirely. 
Similarly, we prove the following partial analogue for genus $0$ coverings: 
%
\begin{theorem}
\label{thm:localch0}
With the exception\footnote{The nonrealizability of $[2^*] [1 , 3^*] [ 2,2, 6^*]$ has already been proved by Zieve.} of $[2^*] [1 , 3^*] [2, 2, 6^*]$, every family of almost-regular ramification data of genus $0$ with $\varepsilon \leq 6$ is realizable in infinitely many degrees. 
Moreover,  for families 
of type $[2,2,2,2]$ or $[3,3,3]$, the same assertion holds with $\varepsilon \leq 10$. 
\end{theorem}
In genus $0$, we use the weaker notion of {\emph{quasi-local changes}}, which allows the error to grow within a family linearly with the degree, and replace 
families of regular dessins on the torus by families of \emph{regular spherical type}, which are quotients of regular dessins by symmetries of the torus, see Section \ref{sec:quasi-local}. 

The outline of the paper is as follows.  We first give some background information on dessins.  Next we explain our theory of almost-regular dessins and local changes. Then we prove the above results, first for genus 1, and then genus 0. \\
\\
Acknowledgements: 
We thank Nir Lazorovich for the help in proving Lemma \ref{notstrip} and Konstantin Golubyev for helpful discussions. We thank Michael Zieve for introducing the composition obstruction used in Conjecture \ref{conj:changes}, 
and other helpful discussions. This work is partially supported by the Israel Science Foundation (grant No. 577/15) and the United States-Israel Binational Science Foundation (grant No.~2014173). 

\section{Preliminaries}
\subsection{Dessins}
\label{sect:dessins}
We give a very brief introduction into standard results on coverings and dessins. For more, see e.g.\ \cite{LZ}.
Let $f:R\to \mathbb{P}^1$ be a degree  $n$ covering of compact connected Riemann surfaces with ordered branch point set $S:=\{p_1,\ldots,p_r\}$ and without loss of generality assume  $p_r=\infty$. 
Let $p_0\in \mathbb{P}^1\setminus S$. Choose paths $\gamma_i$ from $p_0$ to $p_i$, for $i=1,...,r-1$, ordered counter-clockwise. The union $\Gamma$ of  $\gamma_i$, $i=1,\ldots,r-1$ is called a star tree (cf.\ \cite[Section 6.1]{LZ}). A \emph{dessin} for $f$ is the inverse image $f^{-1}(\Gamma)$, viewed as a vertex-colored graph on $R$, with vertex set $f^{-1}(S\cup\{p_0\})$ and $f$ as a coloring function on the vertices, that is, we add the label $i$ to each vertex whose image under $f$ is $p_i$, for $i=1,\ldots,r-1$. 
We follow convention and leave out the color $p_0$ in drawings. 
Classically the term dessin is often restricted to the case $r=3$, with $S=\{0,1,\infty\}$ and $\Gamma=[0,1]$, described in the introduction. 
Two dessins $D_1$ and $D_2$ with the same branch point set, or more generally two colored graphs on surfaces $R_1$ and $R_2$, are {\em equivalent} if there exists an orientation-preserving homeomorphism $R_1\to R_2$ under which $D_1$ and $D_2$ are isomorphic as colored graphs. With this equivalence relation, the choice of paths from $p_0$ to $p_i$ does not matter.

The graph-theoretical information given by a dessin embedded in a surface $R$ encodes all ``relevant" information about the underlying covering $f$; in fact, this information, together with the branch point set, determines the covering up to isomorphism. In particular, the ramification type of $f$ can be recovered from the dessin. Namely, for any given branch point $p_i$, $i\in\{1,...,r-1\}$, the corresponding partition in the ramification type of $f$ is simply the multiset of vertex degrees of points in $f^{-1}(p_i)$ in the dessin. 
Similarly, for the last branch point $p_r$, vertex degrees should be replaced by degrees of faces in the dessin, where the degree of a face is defined as the number of vertices of any given label which are adjacent to the face. 
Note that every face has the same number of vertices of each label.
The Hurwitz existence problem is therefore equivalent to an existence problem of dessins with a given corresponding ramification type. 

We conclude by noting that the graph-theoretical structure of a dessin is of course independent of the position of the branch points in $\mathbb{P}^1$, and that choosing a different ordering of the branch point set $\{p_1,...,p_r\}$ used for the definition of a dessin corresponds to a duality of graphs (switching vertex colors, and possibly interchanging faces and vertices of a certain label. To answer existence questions, we may choose any ordering of the partitions in a given ramification type.

\subsection{almost-regular families of dessins and local changes\label{loc_change}}
A \emph{regular ramification data} is a ramification data containing only a single entry, up to multiplicity, in each multiset.

A {\it regular dessin} is a dessin whose ramification type is regular. In particular, any dessin corresponding to a Galois covering is a regular dessin.
The Riemann--Hurwitz formula implies that there are four genus $1$ regular ramification data:   
$[2^*] [4^*]^2$; 
$[2^*] [3^*] [6^*]$; 
$[2^*]^4$;
and 
$[3^*]^3$, where the exponent denotes the number of occurrences of a tuple in the ramification data. 
These are realized in figure \ref{standard} below as \emph{regular tilings on a torus}, that is, as a quotient of one of the regular tilings of the plane with specified colorings.  Here we identify the top and bottom, as well as left and right sides of the square to get a torus in the usual way. We note that \cite[Proposition 5.1.1]{GS} implies that every dessin with regular ramification type\footnote{In \cite{GS}, two dessins, or more generally normal graphs, with the same ramification type are called combinatorially equivalent and for regular ramification data this is shown to imply topological equivalence} is equivalent to a tiling of $R$ by regular hexagons.

 %


\begin{figure}[h!] 
 \centering
     \begin{subfigure}[t]{.6\textwidth}
     \centering
    \begin{subfigure}[t]{.3\textwidth}
          \begin{tikzpicture}[scale=.58]
 \draw (-4,-4) coordinate(A1) -- (-4,4) coordinate(B1) ; 
 \draw (-2,-4) coordinate(A2) -- (-2,4)  coordinate(B2) ; 
 \draw (0,-4) coordinate(A3) -- (0,4) coordinate(B3)  ; 
 \draw (2,-4) coordinate(A4)  -- (2,4) coordinate(B4)  ; 
 \draw (4,-4) coordinate(A5) -- (4,4) coordinate(B5) ; 
 \draw (-4,-4) coordinate(A1) -- (4,-4) coordinate(A5)  ; 
\draw (-4,-2) coordinate (B6)  -- (4,-2) coordinate(A6) ; 
 \draw (-4,0) coordinate (B7) -- (4,0) coordinate(A7)  ; 
 \draw (-4,2) coordinate (B8) -- (4,2) coordinate(A8) ; 
 \draw (-4,4) coordinate (B1) -- (4,4) coordinate (B5) ; 
 
  \fill[red] (intersection of A1--B1 and B1--B5) circle (4pt); 
  \fill[red] (intersection of A2--B2 and B1--B5) circle (4pt); 
   \fill[red] (intersection of A3--B3 and B1--B5) circle (4pt); 
  \fill[red] (intersection of A4--B4 and B1--B5) circle (4pt); 
   \fill[red] (intersection of A5--B5 and B1--B5) circle (4pt); 
   
  \fill[red] (intersection of A1--B1 and B8--A8) circle (4pt); 
  \fill[red] (intersection of A2--B2 and B8--A8) circle (4pt); 
   \fill[red] (intersection of A3--B3 and B8--A8) circle (4pt); 
  \fill[red] (intersection of A4--B4 and B8--A8) circle (4pt); 
   \fill[red] (intersection of A5--B5 and B8--A8) circle (4pt);    
   
   \fill[red] (intersection of A1--B1 and B7--A7) circle (4pt); 
  \fill[red] (intersection of A2--B2 and B7--A7) circle (4pt); 
   \fill[red] (intersection of A3--B3 and B7--A7) circle (4pt); 
  \fill[red] (intersection of A4--B4 and B7--A7) circle (4pt); 
   \fill[red] (intersection of A5--B5 and B7--A7) circle (4pt);   

  \fill[red] (intersection of A1--B1 and B6--A6) circle (4pt); 
  \fill[red] (intersection of A2--B2 and B6--A6) circle (4pt); 
   \fill[red] (intersection of A3--B3 and B6--A6) circle (4pt); 
  \fill[red] (intersection of A4--B4 and B6--A6) circle (4pt); 
   \fill[red] (intersection of A5--B5 and B6--A6) circle (4pt);      
   
   \fill[red] (intersection of A1--B1 and A1--A5) circle (4pt); 
  \fill[red] (intersection of A2--B2 and A1--A5) circle (4pt); 
   \fill[red] (intersection of A3--B3 and A1--A5) circle (4pt); 
  \fill[red] (intersection of A4--B4 and A1--A5) circle (4pt); 
   \fill[red] (intersection of A5--B5 and A1--A5) circle (4pt);  
   
  \fill[blue] (intersection of B1--B5 and -3,-4-- -3,4) circle (4pt);  
   \fill[blue] (intersection of B1--B5 and -1,-4-- -1,4) circle (4pt); 
  \fill[blue] (intersection of B1--B5 and 1,-4-- 1,4) circle (4pt); 
   \fill[blue] (intersection of B1--B5 and 3,-4-- 3,4) circle (4pt); 
   
  \fill[blue] (intersection of B8--A8 and -3,-4-- -3,4) circle (4pt);  
   \fill[blue] (intersection of B8--A8 and -1,-4-- -1,4) circle (4pt); 
  \fill[blue] (intersection of B8--A8 and 1,-4-- 1,4) circle (4pt); 
   \fill[blue] (intersection of B8--A8 and 3,-4-- 3,4) circle (4pt);    
   
   \fill[blue] (intersection of B7--A7 and -3,-4-- -3,4) circle (4pt);  
   \fill[blue] (intersection of B7--A7 and -1,-4-- -1,4) circle (4pt); 
  \fill[blue] (intersection of B7--A7 and 1,-4-- 1,4) circle (4pt); 
   \fill[blue] (intersection of B7--A7 and 3,-4-- 3,4) circle (4pt);  
   
   \fill[blue] (intersection of B6--A6 and -3,-4-- -3,4) circle (4pt);  
   \fill[blue] (intersection of B6--A6 and -1,-4-- -1,4) circle (4pt); 
  \fill[blue] (intersection of B6--A6 and 1,-4-- 1,4) circle (4pt); 
   \fill[blue] (intersection of B6--A6 and 3,-4-- 3,4) circle (4pt);    

  \fill[blue] (intersection of A1--A5 and -3,-4-- -3,4) circle (4pt);  
   \fill[blue] (intersection of A1--A5 and -1,-4-- -1,4) circle (4pt); 
  \fill[blue] (intersection of A1--A5 and 1,-4-- 1,4) circle (4pt); 
   \fill[blue] (intersection of A1--A5 and 3,-4-- 3,4) circle (4pt); 
   
    \fill[blue] (intersection of A1--B1 and -4,3-- 4,3) circle (4pt); 
   \fill[blue] (intersection of A2--B2 and -4,3-- 4,3) circle (4pt);  
     \fill[blue] (intersection of A3--B3 and -4,3-- 4,3) circle (4pt);
      \fill[blue] (intersection of A4--B4 and -4,3-- 4,3) circle (4pt);      
      \fill[blue] (intersection of A5--B5 and -4,3-- 4,3) circle (4pt); 
      
      \fill[blue] (intersection of A1--B1 and -4,1-- 4,1) circle (4pt); 
   \fill[blue] (intersection of A2--B2 and -4,1-- 4,1) circle (4pt);  
     \fill[blue] (intersection of A3--B3 and -4,1-- 4,1) circle (4pt);
      \fill[blue] (intersection of A4--B4 and -4,1-- 4,1) circle (4pt);      
      \fill[blue] (intersection of A5--B5 and -4,1-- 4,1) circle (4pt); 
      
        \fill[blue] (intersection of A1--B1 and -4,-1-- 4,-1) circle (4pt); 
   \fill[blue] (intersection of A2--B2 and -4,-1-- 4,-1) circle (4pt);  
     \fill[blue] (intersection of A3--B3 and -4,-1-- 4,-1) circle (4pt);
      \fill[blue] (intersection of A4--B4 and -4,-1-- 4,-1) circle (4pt);      
      \fill[blue] (intersection of A5--B5 and -4,-1-- 4,-1) circle (4pt); 
      
       \fill[blue] (intersection of A1--B1 and -4,-3-- 4,-3) circle (4pt); 
   \fill[blue] (intersection of A2--B2 and -4,-3-- 4,-3) circle (4pt);  
     \fill[blue] (intersection of A3--B3 and -4,-3-- 4,-3) circle (4pt);
      \fill[blue] (intersection of A4--B4 and -4,-3-- 4,-3) circle (4pt);      
      \fill[blue] (intersection of A5--B5 and -4,-3-- 4,-3) circle (4pt);
   
\end{tikzpicture} 
         \caption{$[2^{32}][4^{16}][4^{16}]$}
          \label{fig:A}
          \end{subfigure}
     \hfill
 \begin{subfigure}[t]{.3\textwidth}
       \begin{tikzpicture}[scale=.4]
\draw[dashed] (-6,-6) --(6,-6) ; 
\draw[dashed] (-6,6) --(6,6); 
\draw[dashed](-6,-6)--(-6,6); 
\draw[dashed](6,-6)--(6,6);

\draw (-6,5) -- (-4,6); 
\draw(-6,3) -- (-6,5); 
\draw (-4,6) -- (intersection of -4,6--6,1 and -6,3--0,6); 
\draw (-6,3) -- (intersection of -6,1--4,6 and -6,3--6,-3); 
\draw (intersection of -6,1--4,6 and -6,3--6,-3)--(intersection of -6,1--4,6 and -6,5--6,-1); 
\draw (intersection of -6,1--4,6 and -6,5--6,-1)-- (intersection of -4,6--6,1 and -6,3--0,6);
\draw(intersection of -4,6--6,1 and -6,3--0,6)--(intersection of 0,6--6,3 and -6,3--0,6);
\draw (0,6)--(intersection of -0,6--6,3 and -6,1--4,6);
\draw (intersection of 0,6--6,3 and -6,1--4,6)-- (4,6); 
\draw (4,6) --(6,5); 
\draw(6,5) --(6,3); 
\draw (2,5) -- (2,3); 
\draw (-2,3) -- (intersection of -6,5--6,-1 and -6,-1--6,5); 
\draw (2,3) -- (intersection of -6,5--6,-1 and -6,-1--6,5); 
\draw(2,3)-- (intersection of -4,6--6,1 and 6,3---6,-3); 
\draw(6,3)-- (intersection of -4,6--6,1 and 6,3---6,-3); 
\draw (intersection of -6,3--6,-3 and -6,1--4,6) -- (intersection of -6,-1--6,5 and -6,1--6,-5); 
\draw( intersection of -6,5--6,-1 and -6,-1--6,5) --(intersection of -6,-3--6,3 and -6,3--6,-3);
\draw (intersection of -4,6--6,1 and 6,3---6,-3)--(intersection of -6,-5--6,1 and -6,5--6,-1);
\draw (-6,-1)--(intersection of -6,-1--6,5 and -6,1--6,-5);
\draw (intersection of -6,-1--6,5 and -6,1--6,-5)--(-2,-1);
\draw(-2,-1) -- (intersection of -6,-3--6,3 and -6,3--6,-3);
\draw(2,-1) -- (intersection of -6,-3--6,3 and -6,3--6,-3);
\draw (2,-1) -- (intersection of -6,-5--6,1 and -6,5--6,-1);
\draw (intersection of -6,-5--6,1 and -6,5--6,-1)-- (6,-1);
\draw (-6,-1) -- (-6,-3);
\draw (intersection of -6,1--6,-5 and -6,-3--6,3) -- (intersection of -6,-5--6,1 and -6,-1--4,-6);
\draw(2,-1)--(2,-3);
\draw (6, -1) -- (6,-3); 
\draw (-6,-3) -- (-4,-4);
\draw (-4,-4)--(-2,-3);
\draw (-2,-3) --(0,-4);
\draw(0,-4) --(2,-3);
\draw (2,-3) --(4,-4);
\draw (4,-4) --(6,-3);
\draw (-4,-4) --(-4,-6);
\draw (0,-4) --(0,-6) ; 
\draw (4,-4) -- (4, -6); 

\fill[blue] (-4,6) circle (4pt);
\fill[blue] (0,6) circle (4pt);
\fill[blue] (4,6) circle (4pt);
\fill[blue] (-6,3) circle (4pt);
\fill[blue] (-2,3) circle (4pt);
\fill[blue] (2,3) circle (4pt);
\fill[blue] (6,3) circle (4pt);
\fill[blue] (-4,0) circle (4pt);
\fill[blue] (0,0) circle (4pt);
\fill[blue] (4,0) circle (4pt);
\fill[blue] (-6,-3) circle (4pt);
\fill[blue] (-2,-3) circle (4pt);
\fill[blue] (2,-3) circle (4pt);
\fill[blue] (6,-3) circle (4pt);
\fill[blue] (-6,5) circle (4pt);
\fill[blue] (-2,5) circle (4pt);
\fill[blue] (2,5) circle (4pt);
\fill[blue] (6,5) circle (4pt);
\fill[blue] (-4,2) circle (4pt);
\fill[blue] (0,2) circle (4pt);
\fill[blue] (4,2) circle (4pt);
\fill[blue] (-6,-1) circle (4pt);
\fill[blue] (-2,-1) circle (4pt);
\fill[blue] (2,-1) circle (4pt);
\fill[blue] (6,-1) circle (4pt);
\fill[blue] (-4,-4) circle (4pt);
\fill[blue] (0,-4) circle (4pt);
\fill[blue] (4,-4) circle (4pt);

\fill[red] (-6,4) circle (4pt);
\fill[red] (-2,4) circle (4pt);
\fill[red] (2,4) circle (4pt);
\fill[red] (6,4) circle (4pt);
\fill[red] (-4,1) circle (4pt);
\fill[red] (0,1) circle (4pt);
\fill[red] (4,1) circle (4pt);
\fill[red] (-6,-2) circle (4pt);
\fill[red] (-2,-2) circle (4pt);
\fill[red] (2,-2) circle (4pt);
\fill[red] (6,-2) circle (4pt);
\fill[red] (-4,-5) circle (4pt);
\fill[red] (0,-5) circle (4pt);
\fill[red] (4,-5) circle (4pt);
\fill[red] (-5,5.5) circle (4pt);
\fill[red] (-1,5.5) circle (4pt);
\fill[red] (3,5.5) circle (4pt);
\fill[red] (-3,2.5) circle (4pt);
\fill[red] (1,2.5) circle (4pt);
\fill[red] (5,2.5) circle (4pt);
\fill[red] (-5,-.5) circle (4pt);
\fill[red] (-1,-.5) circle (4pt);
\fill[red] (3,-.5) circle (4pt);
\fill[red] (-3,-3.5) circle (4pt);
\fill[red] (1,-3.5) circle (4pt);
\fill[red] (5,-3.5) circle (4pt);
\fill[red] (-3,5.5) circle (4pt);
\fill[red] (1,5.5) circle (4pt);
\fill[red] (5,5.5) circle (4pt);
\fill[red] (-5,2.5) circle (4pt);
\fill[red] (-1,2.5) circle (4pt);
\fill[red] (3,2.5) circle (4pt);
\fill[red] (-3,-.5) circle (4pt);
\fill[red] (1,-.5) circle (4pt);
\fill[red] (5,-.5) circle (4pt);
\fill[red] (-5,-3.5) circle (4pt);
\fill[red] (-1,-3.5) circle (4pt);
\fill[red] (3,-3.5) circle (4pt);

\end{tikzpicture} 
     \caption{$[2^{36}][3^{24}][6^{12}]$}
          \label{fig:B}
          \end{subfigure} 
\end{subfigure}

\centering
   \begin{subfigure}[t]{.6\textwidth}
    \begin{subfigure}[t]{0.3\textwidth}
          \begin{tikzpicture}[scale=.4]
\draw[dashed] (-6,-6) --(6,-6) ; 
\draw[dashed] (-6,6) --(6,6); 
\draw[dashed](-6,-6)--(-6,6); 
\draw[dashed](6,-6)--(6,6);

\draw (-6,5) -- (-4,6); 
\draw(-6,3) -- (-6,5); 
\draw (-4,6) -- (intersection of -4,6--6,1 and -6,3--0,6); 
\draw (-6,3) -- (intersection of -6,1--4,6 and -6,3--6,-3); 
\draw (intersection of -6,1--4,6 and -6,3--6,-3)--(intersection of -6,1--4,6 and -6,5--6,-1); 
\draw (intersection of -6,1--4,6 and -6,5--6,-1)-- (intersection of -4,6--6,1 and -6,3--0,6);
\draw(intersection of -4,6--6,1 and -6,3--0,6)--(intersection of 0,6--6,3 and -6,3--0,6);
\draw (0,6)--(intersection of -0,6--6,3 and -6,1--4,6);
\draw (intersection of 0,6--6,3 and -6,1--4,6)-- (4,6); 
\draw (4,6) --(6,5); 
\draw(6,5) --(6,3); 
\draw (2,5) -- (2,3); 
\draw (-2,3) -- (intersection of -6,5--6,-1 and -6,-1--6,5); 
\draw (2,3) -- (intersection of -6,5--6,-1 and -6,-1--6,5); 
\draw(2,3)-- (intersection of -4,6--6,1 and 6,3---6,-3); 
\draw(6,3)-- (intersection of -4,6--6,1 and 6,3---6,-3); 
\draw (intersection of -6,3--6,-3 and -6,1--4,6) -- (intersection of -6,-1--6,5 and -6,1--6,-5); 
\draw( intersection of -6,5--6,-1 and -6,-1--6,5) --(intersection of -6,-3--6,3 and -6,3--6,-3);
\draw (intersection of -4,6--6,1 and 6,3---6,-3)--(intersection of -6,-5--6,1 and -6,5--6,-1);
\draw (-6,-1)--(intersection of -6,-1--6,5 and -6,1--6,-5);
\draw (intersection of -6,-1--6,5 and -6,1--6,-5)--(-2,-1);
\draw(-2,-1) -- (intersection of -6,-3--6,3 and -6,3--6,-3);
\draw(2,-1) -- (intersection of -6,-3--6,3 and -6,3--6,-3);
\draw (2,-1) -- (intersection of -6,-5--6,1 and -6,5--6,-1);
\draw (intersection of -6,-5--6,1 and -6,5--6,-1)-- (6,-1);
\draw (-6,-1) -- (-6,-3);
\draw (intersection of -6,1--6,-5 and -6,-3--6,3) -- (intersection of -6,-5--6,1 and -6,-1--4,-6);
\draw(2,-1)--(2,-3);
\draw (6, -1) -- (6,-3); 
\draw (-6,-3) -- (-4,-4);
\draw (-4,-4)--(-2,-3);
\draw (-2,-3) --(0,-4);
\draw(0,-4) --(2,-3);
\draw (2,-3) --(4,-4);
\draw (4,-4) --(6,-3);
\draw (-4,-4) --(-4,-6);
\draw (0,-4) --(0,-6) ; 
\draw (4,-4) -- (4, -6); 

\fill[red] (-6,4) circle (4pt);
\fill[red] (-2,4) circle (4pt);
\fill[red] (2,4) circle (4pt);
\fill[red] (6,4) circle (4pt);
\fill[red] (-4,1) circle (4pt);
\fill[red] (0,1) circle (4pt);
\fill[red] (4,1) circle (4pt);
\fill[red] (-6,-2) circle (4pt);
\fill[red] (-2,-2) circle (4pt);
\fill[red] (2,-2) circle (4pt);
\fill[red] (6,-2) circle (4pt);
\fill[red] (-4,-5) circle (4pt);
\fill[red] (0,-5) circle (4pt);
\fill[red] (4,-5) circle (4pt);

\fill[blue] (-5,5.5) circle (4pt);
\fill[blue] (-1,5.5) circle (4pt);
\fill[blue] (3,5.5) circle (4pt);
\fill[blue] (-3,2.5) circle (4pt);
\fill[blue] (1,2.5) circle (4pt);
\fill[blue] (5,2.5) circle (4pt);
\fill[blue] (-5,-.5) circle (4pt);
\fill[blue] (-1,-.5) circle (4pt);
\fill[blue] (3,-.5) circle (4pt);
\fill[blue] (-3,-3.5) circle (4pt);
\fill[blue] (1,-3.5) circle (4pt);
\fill[blue] (5,-3.5) circle (4pt);

\fill[green] (-3,5.5) circle (4pt);
\fill[green] (1,5.5) circle (4pt);
\fill[green] (5,5.5) circle (4pt);
\fill[green] (-5,2.5) circle (4pt);
\fill[green] (-1,2.5) circle (4pt);
\fill[green] (3,2.5) circle (4pt);
\fill[green] (-3,-.5) circle (4pt);
\fill[green] (1,-.5) circle (4pt);
\fill[green] (5,-.5) circle (4pt);
\fill[green] (-5,-3.5) circle (4pt);
\fill[green] (-1,-3.5) circle (4pt);
\fill[green] (3,-3.5) circle (4pt);
\end{tikzpicture} 
           
          \caption{$[2^{12}][2^{12}][2^{12}][2^{12}]$}
          \label{fig:C}
\end{subfigure} 
\hfill
 \begin{subfigure}[t]{0.3\textwidth}
 \begin{tikzpicture}[scale=.4]
\draw[dashed] (-6,-6) --(6,-6) ; 
\draw[dashed] (-6,6) --(6,6); 
\draw[dashed](-6,-6)--(-6,6); 
\draw[dashed](6,-6)--(6,6);

\draw (-6,5) -- (-4,6); 
\draw(-6,3) -- (-6,5); 
\draw (-4,6) -- (intersection of -4,6--6,1 and -6,3--0,6); 
\draw (-6,3) -- (intersection of -6,1--4,6 and -6,3--6,-3); 
\draw (intersection of -6,1--4,6 and -6,3--6,-3)--(intersection of -6,1--4,6 and -6,5--6,-1); 
\draw (intersection of -6,1--4,6 and -6,5--6,-1)-- (intersection of -4,6--6,1 and -6,3--0,6);
\draw(intersection of -4,6--6,1 and -6,3--0,6)--(intersection of 0,6--6,3 and -6,3--0,6);
\draw (0,6)--(intersection of -0,6--6,3 and -6,1--4,6);
\draw (intersection of 0,6--6,3 and -6,1--4,6)-- (4,6); 
\draw (4,6) --(6,5); 
\draw(6,5) --(6,3); 
\draw (2,5) -- (2,3); 
\draw (-2,3) -- (intersection of -6,5--6,-1 and -6,-1--6,5); 
\draw (2,3) -- (intersection of -6,5--6,-1 and -6,-1--6,5); 
\draw(2,3)-- (intersection of -4,6--6,1 and 6,3---6,-3); 
\draw(6,3)-- (intersection of -4,6--6,1 and 6,3---6,-3); 
\draw (intersection of -6,3--6,-3 and -6,1--4,6) -- (intersection of -6,-1--6,5 and -6,1--6,-5); 
\draw( intersection of -6,5--6,-1 and -6,-1--6,5) --(intersection of -6,-3--6,3 and -6,3--6,-3);
\draw (intersection of -4,6--6,1 and 6,3---6,-3)--(intersection of -6,-5--6,1 and -6,5--6,-1);
\draw (-6,-1)--(intersection of -6,-1--6,5 and -6,1--6,-5);
\draw (intersection of -6,-1--6,5 and -6,1--6,-5)--(-2,-1);
\draw(-2,-1) -- (intersection of -6,-3--6,3 and -6,3--6,-3);
\draw(2,-1) -- (intersection of -6,-3--6,3 and -6,3--6,-3);
\draw (2,-1) -- (intersection of -6,-5--6,1 and -6,5--6,-1);
\draw (intersection of -6,-5--6,1 and -6,5--6,-1)-- (6,-1);
\draw (-6,-1) -- (-6,-3);
\draw (intersection of -6,1--6,-5 and -6,-3--6,3) -- (intersection of -6,-5--6,1 and -6,-1--4,-6);
\draw(2,-1)--(2,-3);
\draw (6, -1) -- (6,-3); 
\draw (-6,-3) -- (-4,-4);
\draw (-4,-4)--(-2,-3);
\draw (-2,-3) --(0,-4);
\draw(0,-4) --(2,-3);
\draw (2,-3) --(4,-4);
\draw (4,-4) --(6,-3);
\draw (-4,-4) --(-4,-6);
\draw (0,-4) --(0,-6) ; 
\draw (4,-4) -- (4, -6); 

\fill[blue] (-4,6) circle (4pt);
\fill[blue] (0,6) circle (4pt);
\fill[blue] (4,6) circle (4pt);
\fill[blue] (-6,3) circle (4pt);
\fill[blue] (-2,3) circle (4pt);
\fill[blue] (2,3) circle (4pt);
\fill[blue] (6,3) circle (4pt);
\fill[blue] (-4,0) circle (4pt);
\fill[blue] (0,0) circle (4pt);
\fill[blue] (4,0) circle (4pt);
\fill[blue] (-6,-3) circle (4pt);
\fill[blue] (-2,-3) circle (4pt);
\fill[blue] (2,-3) circle (4pt);
\fill[blue] (6,-3) circle (4pt);

\fill[red] (-6,5) circle (4pt);
\fill[red] (-2,5) circle (4pt);
\fill[red] (2,5) circle (4pt);
\fill[red] (6,5) circle (4pt);
\fill[red] (-4,2) circle (4pt);
\fill[red] (0,2) circle (4pt);
\fill[red] (4,2) circle (4pt);
\fill[red] (-6,-1) circle (4pt);
\fill[red] (-2,-1) circle (4pt);
\fill[red] (2,-1) circle (4pt);
\fill[red] (6,-1) circle (4pt);
\fill[red] (-4,-4) circle (4pt);
\fill[red] (0,-4) circle (4pt);
\fill[red] (4,-4) circle (4pt);
\end{tikzpicture} 
 \caption{$[3^{12}][3^{12}][3^{12}]$} 
 \label{fig:D}
 \end{subfigure}
 \end{subfigure}
 
 \caption{The 4 regular types of dessins on a torus}    
 \label{standard}    
 \end{figure}
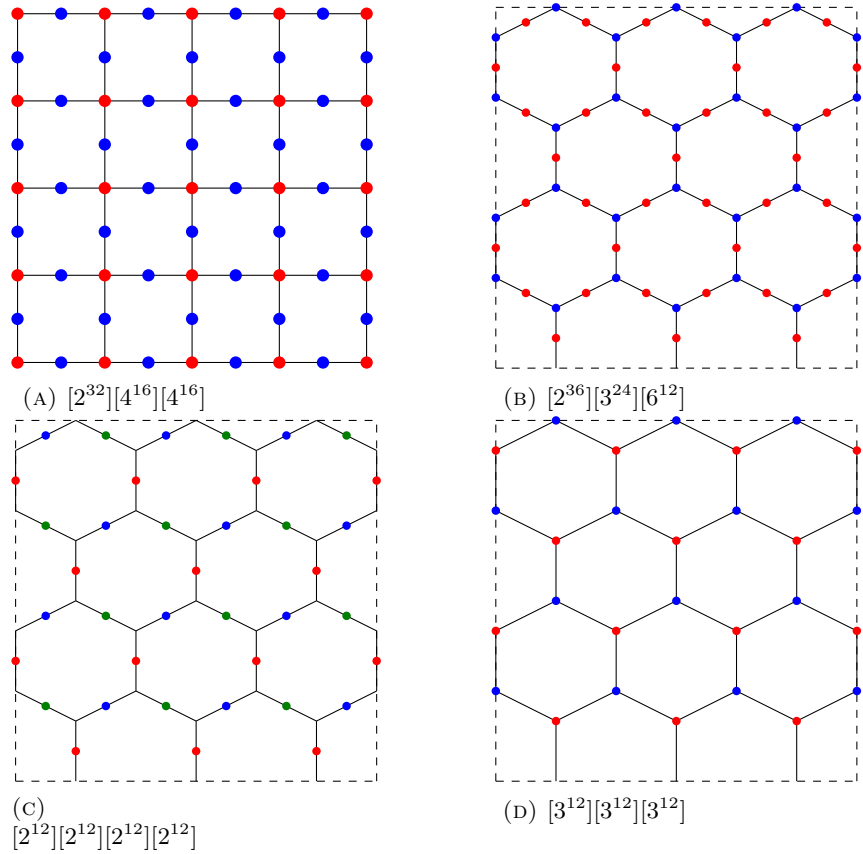
 
 
%
%


As in Section \ref{sec:intro}, we consider only almost-regular families $(T_i)_{i\in \NN}$ whose irregular entries are the same for all $i$, and the higher degree members of this family are obtained by adding 
regular entries.
\begin{remark}\label{rem:general-almost} 
The results in this paper hold more general for the following definition of almost-regular ramification data. Call a family $T = (T_i)_{i\in\mathbb N}$ of ramification data with $r$ branch points of {\it (generalized) almost-regular of type} $[k_1,\ldots,k_r]$ 
{\it of error at most $\eps$}, 
if $\deg T_i\ra\infty$ as $i\ra\infty$ and  the sum of errors $\sum_{j=1}^r\eps_{i,j}$ is at most $\eps$ for every $i\in\mathbb N$, where the {\it error} $\eps_{i,j}$ is the sum of (``irregular") entries in $T_{i,j}$ different from $k_j$ for $j=1,\ldots,r$. 
 There is no loss of generality in restricting to our definition, since for a given $\eps$, the number of almost-regular ramification types of a fixed degree with error at most $\eps$ is bounded by a constant independent of the degree of the ramification type. We may then restrict to suitable subfamilies all of whose members have the same irregular entries. 
\end{remark} 


A \emph{change to a dessin} consists of adding or removing an edge (connecting a preimage of some branch point to a preimage of the base point) or an isolated vertex of any color. 

\begin{definition}[Local Changes to Families of Dessins]
\label{def:close}
Let $(C_i)_{i\in \NN}$ and $(D_i)_{i\in \NN}$ be two families of dessins on a surface $R$, all with the same set of branch points. 
 We say that $(C_i)_i$ is realizable as \emph{local changes} to $(D_i)_i$ if there exists a constant $c\in \NN$ such that for all $i\in \NN$, applying at most $c$ changes to $C_i$ gives a dessin equivalent to $D_i$. 
%
If in addition $(D_i)_i$ is a family of regular dessins, we call $(C_i)_i$ a family of \emph{almost-regular} dessins.
\end{definition}
Figure \ref{[2^*][2^*][1,1,2^*][4,2^*]} shows a family of dessins realizable as local changes to $[2^*][2^*][2^*][2^*]$. 
\begin{figure}[h!]
\begin{tikzpicture}[scale=.5]
\draw[dashed] (-6,-6) --(6,-6) ; 
\draw[dashed] (-6,6) --(6,6); 
\draw[dashed](-6,-6)--(-6,6); 
\draw[dashed](6,-6)--(6,6);

\draw (-6,5) -- (-4,6); 
\draw(-6,3) -- (-6,5); 
\draw (-4,6) -- (intersection of -4,6--6,1 and -6,3--0,6); 
\draw (-6,3) -- (intersection of -6,1--4,6 and -6,3--6,-3); 
\draw (intersection of -6,1--4,6 and -6,3--6,-3)--(intersection of -6,1--4,6 and -6,5--6,-1); 
\draw (intersection of -6,1--4,6 and -6,5--6,-1)-- (intersection of -4,6--6,1 and -6,3--0,6);
\draw(intersection of -4,6--6,1 and -6,3--0,6)--(intersection of 0,6--6,3 and -6,3--0,6);
\draw (0,6)--(intersection of -0,6--6,3 and -6,1--4,6);
\draw (intersection of 0,6--6,3 and -6,1--4,6)-- (4,6); 
\draw (4,6) --(6,5); 
\draw(6,5) --(6,3); 
\draw (2,5) -- (2,3); 
\draw (-2,3) -- (intersection of -6,5--6,-1 and -6,-1--6,5); 
\draw (2,3) -- (intersection of -6,5--6,-1 and -6,-1--6,5); 
\draw(2,3)-- (intersection of -4,6--6,1 and 6,3---6,-3); 
\draw(6,3)-- (intersection of -4,6--6,1 and 6,3---6,-3); 
\draw (intersection of -6,3--6,-3 and -6,1--4,6) -- (intersection of -6,-1--6,5 and -6,1--6,-5); 
\draw (intersection of -4,6--6,1 and 6,3---6,-3)--(intersection of -6,-5--6,1 and -6,5--6,-1);
\draw (-6,-1)--(intersection of -6,-1--6,5 and -6,1--6,-5);
\draw (intersection of -6,-1--6,5 and -6,1--6,-5)--(-2,-1);
\draw(-2,-1) -- (intersection of -6,-3--6,3 and -6,3--6,-3);
\draw(2,-1) -- (intersection of -6,-3--6,3 and -6,3--6,-3);
\draw (2,-1) -- (intersection of -6,-5--6,1 and -6,5--6,-1);
\draw (intersection of -6,-5--6,1 and -6,5--6,-1)-- (6,-1);
\draw (-6,-1) -- (-6,-3);
\draw (intersection of -6,1--6,-5 and -6,-3--6,3) -- (intersection of -6,-5--6,1 and -6,-1--4,-6);
\draw(2,-1)--(2,-3);
\draw (6, -1) -- (6,-3); 
\draw (-6,-3) -- (-4,-4);
\draw (-4,-4)--(-2,-3);
\draw (-2,-3) --(0,-4);
\draw(0,-4) --(2,-3);
\draw (2,-3) --(4,-4);
\draw (4,-4) --(6,-3);
\draw (-4,-4) --(-4,-6);
\draw (0,-4) --(0,-6) ; 
\draw (4,-4) -- (4, -6); 

\fill[red] (-6,4) circle (4pt);
\fill[red] (-2,4) circle (4pt);
\fill[red] (2,4) circle (4pt);
\fill[red] (6,4) circle (4pt);
\fill[red] (-4,1) circle (4pt);
\fill[red] (4,1) circle (4pt);
\fill[red] (-6,-2) circle (4pt);
\fill[red] (-2,-2) circle (4pt);
\fill[red] (2,-2) circle (4pt);
\fill[red] (6,-2) circle (4pt);
\fill[red] (-4,-5) circle (4pt);
\fill[red] (0,-5) circle (4pt);
\fill[red] (4,-5) circle (4pt);

\fill[blue] (-5,5.5) circle (4pt);
\fill[blue] (-1,5.5) circle (4pt);
\fill[blue] (3,5.5) circle (4pt);
\fill[blue] (-3,2.5) circle (4pt);
\fill[blue] (1,2.5) circle (4pt);
\fill[blue] (5,2.5) circle (4pt);
\fill[blue] (-5,-.5) circle (4pt);
\fill[blue] (-1,-.5) circle (4pt);
\fill[blue] (3,-.5) circle (4pt);
\fill[blue] (-3,-3.5) circle (4pt);
\fill[blue] (1,-3.5) circle (4pt);
\fill[blue] (5,-3.5) circle (4pt);

\fill[green] (-3,5.5) circle (4pt);
\fill[green] (1,5.5) circle (4pt);
\fill[green] (5,5.5) circle (4pt);
\fill[green] (-5,2.5) circle (4pt);
\fill[green] (-1,2.5) circle (4pt);
\fill[green] (3,2.5) circle (4pt);
\fill[green] (-3,-.5) circle (4pt);
\fill[green] (1,-.5) circle (4pt);
\fill[green] (5,-.5) circle (4pt);
\fill[green] (-5,-3.5) circle (4pt);
\fill[green] (-1,-3.5) circle (4pt);
\fill[green] (3,-3.5) circle (4pt);

\draw(0,2) --(-2,1) ; 
\draw (0,0)-- (2,1); 
\fill[red] (-2,1) circle (4pt);
\fill[red] (2,1) circle (4pt);
\end{tikzpicture} 
\caption{ $[2^{12}][2^{12}][1,1,2^{11}][4,2^{10}]$}
\label{[2^*][2^*][1,1,2^*][4,2^*]}
\end{figure}
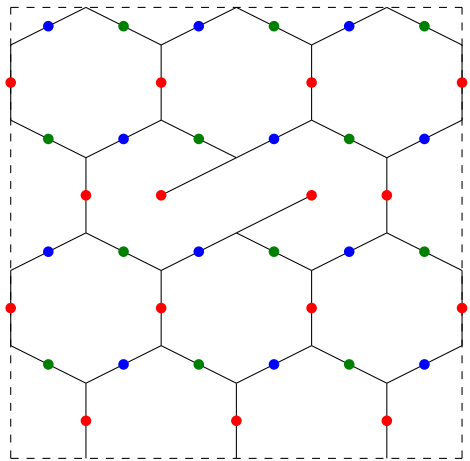
The following lemma shows that Definition \ref{def:close} is independent of the choice of base graph. 
\begin{lemma}
\label{lem:base_curve}
Let $(f_i)_{i\in \NN}$ and $(g_i)_{i\in \NN}$ be two families of coverings from a surface $X$ to $\mathbb{P}^1$, 
with the same branch point set $\{p_1,...,p_r\}\subset\mathbb{P}^1$.
Let $S \subset\mathbb{P}^1$ be a star tree passing through $p_1,...,p_{r-1}$, and let $\tilde{S}$ be a graph on $\mathbb{P}^1$ with vertices $p_1,...,p_{r-1},p_r$.
Denote the dessins $C_i = f_i ^{-1} (S)$ and $\tilde C_i = f_i ^{-1} (\tilde S)$, respectively $D_i = g_i ^{-1} (S)$, and $\tilde D_i = g_i ^{-1} (\tilde S)$ for $i \in \NN$. 
If $(D_i)_i$ is obtained via local changes from $(C_i)_i$, then also $(\tilde{D}_i)_i$ is obtained via local changes from $(\tilde{C}_i)_i$.\footnote{Of course, local changes between arbitrary families of colored graphs are to be defined exactly as for families of dessins in Definition \ref{def:close}.}
\end{lemma}
\begin{proof}
At first fix $i\in\NN$ and let $n_i$ be the degree of $f_i$. 
Then $C_i$ consists of copies of $S$ glued together at some of the ramification points of $f_i$. 
Label the copies of $S$ by $S_j$ for $j\in I_i$. 
If $q_{j,k}\in S_j$ denotes the preimage of $p_k$ under $f_i$ for $k=1,\ldots,r-1$, then the homotopy lifting property implies that in the preimage of $\tilde S$ under $f_i$, the vertices $q_{j,1},\ldots, q_{j,r-1}$ are connected by a graph isomorphic (via $f_i$) to $\tilde S$, for $j\in I_i$. 
Thus the graph  $\tilde{C_i}$ is obtained by replacing each $S_j$ with a copy $\tilde S_j$ of $\tilde S$ for $j\in I_i$.  Note that the gluing of the copies $\tilde S_j$, $j\in J_i$ in $\tilde C_i$ is done at the same vertices as the gluing of the copies $S_j$, $j\in I_i$ in $C_i$. 
Analogously, the same holds for $D_i$ and $\tilde{D}_i$. Denote the copies of $S$ and $\tilde S$ in $D_i$ and $\tilde D_i$ by $T_j$ and $\tilde T_j$ for $j\in J_i$, respectively. 


Since by assumption $(D_i)_i$ is obtained via local changes from $(C_i)_i$, there exist subsets $I'_i\subseteq I_i$ and $J'_i\subseteq J_i$ such that the cardinalities of $I_i\setminus I'_i$ and $J_i\setminus J'_i$ are bounded by a constant independent of $i$, and  
an isomorphism $\phi_i$ between the subgraph $\cup_{j\in I_i}S_j$ and the subgraph  $\cup_{j\in J'_i}T_j$, for every $i\in\NN$. Since $\tilde C_i$ and $\tilde D_i$ are obtained from $C_i$ and $D_i$, respectively, by replacing copies of $S$ with copies of $\tilde S$, the isomorphism $\phi_i$ induces an isomorphism $\psi_i$ between the subgraph $\cup_{j\in I_i'}\tilde S_j$ and the subgraph  $\cup_{j\in J'_i}\tilde T_j$, for every $i\in\NN$. As the number of vertices and edges in the complement of these subgraphs is bounded by a constant independent of $i$,  $(\tilde C_i)_i$ and $(\tilde D_i)_i$ differ by local changes. 
\end{proof}


Note that if the family $(C_i)_i$ of dessins is realizable as local changes to a family $(D_i)_i$, then the corresponding ramification types differ by an absolutely bounded amount of entries, as these correspond to the degrees of vertices and faces.  
In particular, if $(C_i)_i$ is a family of almost-regular dessins,
then also the family of ramification types of $(C_i)_i$ is almost-regular.
The converse of this simple observation can be seen as the second part of Question \ref{conj:existence}. 

\section{Proofs of the main results for $g=1$}
\subsection{Reduction arguments}\label{red}
Here we review some elementary techniques to reduce existence of certain ramification types to other types. 
\subsubsection{Composition}
Composition of a degree-$n$ covering $f:R\to \mathbb{P}^1$ with a non-constant degree-$d$ rational function $g:\mathbb{P}^1\to \mathbb{P}^1$ yields a degree-$nd$ covering $g\circ f:R\to \mathbb{P}^1$. 
For example, the regular ramification type $[2^k][3^{2k/3}][6^{k/3}]$ may be realized by composing a covering of ramification type $[3^{k/3}]^3$ and suitable choice of branch points with a covering of ramification type $[2][2]$ as follows. Let $f: \mathbb{T} \to \mathbb{P}^1$ be a covering from the torus to the sphere, of ramification type $[3^{k/3}]^3$ and branch point set $\{\infty, 1,-1\}$. Define $g:\mathbb{P}^1\to \mathbb{P}^1$ by $x\mapsto x^2$. The branch points of $g$ are $0$ and $\infty$, and the only preimages of these are also $0$ and $\infty$, respectively. 
Multiplicativity of ramification indices then shows that the branch points of $g\circ f$ are $0,1$ and $\infty$, and its  ramification type is $[2^k][3^{2k/3}][6^{k/3}]$. 
%

Composition with rational function preserves ``local changes" in the sense of Definition\ \ref{def:close}.
\begin{lemma}
\label{lemma:rattranslates}
Let $(f_i)_{i\in\NN}$ and $(g_i)_{i\in\NN}$ be two families of coverings $R\to\mathbb{P}^1$, with the same branch point set. Let 
$h:\mathbb{P}^1\to \mathbb{P}^1$ a covering. 
Let $C_i, D_i, \tilde C_i, \tilde D_i$ be the dessins corresponding to $f_i, g_i, h\circ f_i,h\circ g_i$, respectively, for $i\in \NN$. 
If $(C_i)_i$ is obtained by local changes from $(D_i)_i$, then $(\tilde C_i)_i$ is obtained by local changes from $(\tilde D_i)$,  $i\in \NN$.
\end{lemma}
\begin{proof}
Let $S:=\{s_1,...,s_k\}$ be the set of finite branch points of $h\circ f_i$ for $i\in \NN$
and $T:=h^{-1}(S)$. Let $\mathcal{S}\subset \mathbb{P}^1$ be a star through $S$. Then $\tilde C_i$ is equivalent and hence can be replaced by the preimage under $f_i$ of the graph $G:=h^{-1}(\mathcal{S})$ with vertex set $T$. 
Similarly, $\tilde D_i$ can be replaced by the preimage under $g_i$ of $G$. Since $(C_i)_i$ is obtained by local changes from $(D_i)_i$, Lemma \ref{lem:base_curve} implies that the family of colored graphs $f_{i}^{-1}(G)$, $i\in \NN$, are obtained by local changes from $g_{i}^{-1}(G)$. 
\end{proof}

\subsubsection{A simple example of a local change}
In addition to the the technique of composing maps, existence of certain ramification data 
can be obtained by preforming certain moves on realizable ramification data. An example is the following lemma which changes only certain entries of a ramification data. We denote by $A_j$ a multiset of entries in the $j$-th partition. 
\begin{lemma}
\label{lemma:addlines}
Let $r\ge 3$. Let $T$ be the ramification data  $[a, A_1][b, A_2][A_3]...[A_4]$. Assume  $D$ is a dessin with ramification type $T$ in which  
there exists a vertex of degree $a$ and label $1$ adjacent to a vertex of degree $b$ and label $2$.
Then the following ramification data is realizable
$${T}^{(k)}:=([a+k , A_1][b+k ,  A_2 ][1^{k}, A_3]...[1^{k}, A_r])\text{ for every }k\in \NN.$$ 
Furthermore, if an infinite family of ramification data $(T_i)_{i\in\NN}$ is realizable by a family of almost-regular dessins, then the same holds for the corresponding family of ramification data $(T_i^{(k)})_{i\in\NN}$.
\end{lemma}
\begin{proof}
Replace a line connecting the two vertices of degree $a$ and $b$ by the the construction in Figure \ref{add_edges}. 
This increases the degree of those two vertices by $k$ each; furthermore, $k$ vertices of degree $1$ are added for each other vertex color, and similarly, $k$ faces of degree $1$ are added.

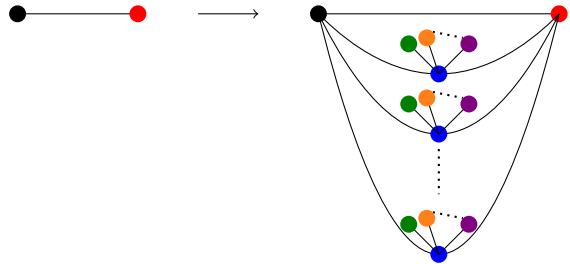
\begin{figure}[h!] 
\centering 
\begin{tikzpicture}[scale=.8] 
\fill[black] (-6,0) circle (4pt);
\draw (-6, 0)-- (-4,0) ; 
\fill[red] (-4,0) circle (4pt);
\draw[->] (-3,0) -- ( -2,0); 
\fill[black] (-1,0) circle (4pt);
\draw (-1, 0)-- (3,0) ; 
\fill[red] (3,0) circle (4pt);
\draw (-1,0) parabola bend (1, -1) (3,0); 
\draw (-1,0) parabola bend (1, -2) (3,0); 
\draw (-1,0) parabola bend (1, -4) (3,0); 
\draw[dotted, thick] (1, -2.25) -- ( 1, -3); 
\fill[blue] (1,-1) circle (4pt);
\fill[blue] (1,-2) circle (4pt);
\fill[blue] (1,-4) circle (4pt);
\draw ( 1, -1) -- (.5, -.5); 
\draw ( 1, -1) -- (1.5, -.5); 
\draw (1, -1) -- (.8 , -.4) ; 
\draw [dotted, thick] (. 9, -.3) -- ( 1.4, -.4); 
\fill[green] (.5,-.5) circle (4pt);
\fill[orange] (.8,-.4) circle (4pt);
\fill[purple] (1.5,-.5) circle (4pt);
\draw ( 1, -2) -- (.5, -1.5); 
\draw ( 1, -2) -- (1.5, -1.5); 
\draw (1, -2) -- (.8 , -1.4) ; 
\draw [dotted, thick] (. 9, -1.3) -- ( 1.4, -1.4); 
\fill[green] (.5,-1.5) circle (4pt);
\fill[orange] (.8,-1.4) circle (4pt);
\fill[purple] (1.5,-1.5) circle (4pt);
\draw ( 1, -4) -- (.5, -3.5); 
\draw ( 1, -4) -- (1.5, -3.5); 
\draw (1, -4) -- (.8 , -3.4) ; 
\draw [dotted, thick] (. 9, -3.3) -- ( 1.4, -3.4); 
\fill[green] (.5,-3.5) circle (4pt);
\fill[orange] (.8,-3.4) circle (4pt);
\fill[purple] (1.5,-3.5) circle (4pt);
\end{tikzpicture} 
\caption{Adding $k$ faces of degree 1, and $k$ degree 1 vertices}
\label{add_edges}
\end{figure} 

The second claim follows, since the transition from $T_i$ to $T_i^{(k)}$ with the above construction adds only a bounded number of edges and vertices.
\end{proof}

\subsection{A method for local changes to dessins on the torus}
The following proposition is our main tool to form dessins on a torus from a single dessin with suitable properties.

We first introduce terminology related to tilings of the plane  with regular polygons. 
Let $\mathbb{A}$ be the $2$-dimensional plane tiled with polygons of diameter bounded from below. A  subset $D$ of polygons is called a {\emph{disk of radius $r$}} if it consists of all polygons of (graph) distance at most $r$ from a single polygon, called the center of the disk. Here, the distance of two polygons is the minimal number of polygons on a path connecting the centers of the polygons.
Consider the induced tiling on the torus obtained from quotienting $\mathbb{A}$ by a lattice with generators the \emph{longitude} and \emph{meridian} of the torus. The parallelogram $P$ formed by longitude and meridian is the \emph{fundamental domain} for the torus.
If a disk of polygons of radius $r$ is completely contained in the inside of the fundamental domain, then the induced graph on the torus {\it contains a disk of radius $r$}. 
A family of graphs  $(D_i)_{i\in\NN}$ on a torus {\it contains disks of arbitrary radius} if for every $r\in \NN$, there exists $i_0\in \NN$ such that $D_i$ contains a disk of radius $r$ for all $i\ge i_0$.

\begin{proposition}
\label{prop:localch1}
Let $c>0$ be a constant. There exists a constant $r=r(c)$, depending only on $c$, with the following property. Let $(T_i)_{i\in\NN}$ be a family of almost-regular genus $1$ ramification data. 
If there exist $k\in \NN$ and a dessin  $C_k$ for $T_k$ which is obtained by at most $c$ changes to a regular dessin containing a disk of radius $r$,  then with the exception of finitely many terms, the family $(T_i)_i$ is realizable by local changes to a family of regular dessins containing disks of arbitrary large radius. 
\end{proposition}

Our strategy is to remove part of the regular tiling and replace it with an almost-regular tiling without changing the rest of the tiles.  
To do so we show there exists a tiling of the torus consisting of any sufficiently large number of hexagons and containing disks of arbitrary large radius. The proof is carried out with hexagons but also applies to squares.   





\begin{lemma}\label{notstrip}  Given $n \in \NN$, the torus can be tiled by a regular dessin consisting of $n$ hexagons.  Moreover, for all $r\in \NN$ and  sufficiently large $n$ compared to $r$, there is such tiling containing  a disk of radius $r$.
\end{lemma} 

\begin{proof} We want to show that we can quotient a hexagon tiling of the plane by a parallelogram lattice containing the centers of $n$ hexagons  
to obtain a tiling of the torus with $n$ hexagons, which  contains a disk of radius $r$ if $n$ is sufficiently large. 

Begin with a tiling of the plane by hexagons. Consider the short exact sequence 
$$1 \to \ZZ^2 \to \ZZ^2 \to \ZZ /n \to 1 .$$ 
Finding a parallelogram containing the centers of $n$ hexagons  is equivalent to finding a map $\ZZ^2 \to \ZZ^2$ of determinant $n$.  
To get a disk of radius $r$ inside the parallelogram for all sufficiently large $n$, 
it suffices to have side lengths tending to infinity and angle bounded away from $0$. 

There are three cases.  If $n = a^2$ is a square,  we use the parallelogram spanned by the columns of the matrix $\left(\begin{smallmatrix} a& 0\\
0 & a \end{smallmatrix} \right) $ and the angle between the column vectors is a right angle. 
If $n$ is not a square,  choose $a$ such that $a^2 < n <(a+1)^2$.  
Then either $a^2$ or $(a+1)^2$ is nearer to $n$. Choose $k$ such that either 
$ k = n-a ^2 \leq a$ or $k = (a+1)^2  -n \leq a.$
In the latter case,  
use the matrix $\left(\begin{smallmatrix} a+1 & k\\
1 & a+1 \end{smallmatrix} \right) $. 
This matrix  has determinant $n$, the length of the column vectors is at least $a\ge \sqrt{n}-1$. 
The angle between the column vectors has cosine  
$$\frac{(a+1)k+a+1}{\sqrt{(a+1)^2+1}\sqrt{(a+1)^2 + k^2}}\leq \frac{k+1}{\sqrt{2(a+1)k}}\leq \frac{1}{\sqrt{2}} +0.1
$$
where the last inequality holds for $a>10$. 
Since the result is clearly bounded away from $1$, the angle is bounded away from $0$. 

In the case where $k = n-a^2$, use the matrix $\left ( \begin{smallmatrix} a +1 & k \\ 1 & a+1 \end{smallmatrix} \right ) $ and a similar but simpler computation yields that cosine of the angle is bounded by $1/\sqrt{2}$ and hence the angle is bounded away from $0$. Picking the parallelogram this way gives a tiling of the torus by $n$ hexagons, and if $n$ is sufficiently large a parallelogram which contains a disk of radius $r$. 
\end{proof} 

\begin{proof}[Proof of Proposition  \ref{prop:localch1}] 
Let $D_k$ be a regular dessin on the torus which without loss of generality contains a disk of radius $r$, consisting of regular polygons, and assume that $C_k$ is a dessin of ramification type $T_k$ which is obtained from $D_k$ by at most $c$ changes. Note that these changes can only affect a bounded number of polygons in $D_k$, depending only on $c$. Consider the fundamental domain $P$ for $D_k$ in the plane. If $r$ is sufficiently large in comparison to $c$, then there exists some parallel to the longitude which intersects none of the polygons affected by the above changes; analogously for the meridian. That is, we may assume without loss, by translating $P$, that none of the changes affect polygons which intersect the meridian or longitude of $P$. 

We can thus extend the tiling of our fundamental domain to a tiling of the plane which matches a tiling by regular polygons outside of this fundamental domain.

Now choose $R\in \NN$ sufficiently large, e.g.\ such that, in the extended tiling above, there is a disk of polygons of radius $R$ which contains the complete fundamental domain $P$ in its interior.
By Lemma \ref{notstrip}, for any sufficiently large $i\in \NN$, we can tile the torus by a regular dessin consisting of exactly $i$ hexagons (resp.~squares), containing a disk of polygons of radius at least $R$. We remove hexagons (resp.~squares) within the radius $r$ (sub-)disk and replace them by the same changes as for the dessin $C_k$.  Since this always requires only the fixed number $c$ of changes, we obtain a dessin $C_i$ of ramification type $T_i$, for all sufficiently large $i$, and such that the family $(C_i)_i$ arises as $\le c$ changes to a family of regular dessins containing disks of arbitrary large radius, cf.~figure \ref{[2^*][2^*][1,1,2^*][4,2^*]}.  
\end{proof}



A key observation for the proof of our main results is that, under the assumptions on $(T_i)_{i\in \NN}$, the implication of Proposition \ref{prop:localch1} is also compatible with composition of maps, as made precise in Lemma \ref{lem:glue2} below. Note the iterative nature of this lemma, which will be made use of in the reduction arguments in the following section.

\begin{lemma}
\label{lem:glue2}
Let $(f_i)_{i\in \NN}$ and $(g_i)_{i\in\NN}$ be  families of coverings $\mathbb{T}\to \mathbb{P}^1$ with $g_i$, $i\in \NN$ Galois. 
Let $(D_i)_{i\in\NN}$ be a family of regular dessins for $(g_i)_{i\in \NN}$ which contains disks of arbitrary radius, and assume the dessins $(C_i)_{i\in\NN}$ for $(f_i)_{i\in\NN}$ are obtained by local changes to $(D_i)_{i\in\NN}$. 
%
Let $(T_j)_{j\in\NN}$ be a family of genus 1 almost-regular ramification data and $h:\mathbb{P}^1\to \mathbb{P}^1$ a fixed covering such that the ramification types of $(h\circ f_i)_{i\in\NN}$ are contained in $(T_j)_{j\in \NN}$. Then $(T_j)_{j\in \NN}$ are realizable in every sufficiently large degree by a family of almost-regular dessins $(\tilde C_j)_{j\in \NN}$.\footnote{We emphasize here that we obtain a statement for {\textit{all}} sufficiently large degrees of some family of ramification types, and not only for those degrees which correspond to a decomposable covering $h\circ f_i$!} 

Furthermore, the family $(\tilde C_j)_{j\in \NN}$ is obtained by local changes to a family of regular dessins containing disks of arbitrary radius. 
\end{lemma}
\begin{proof}
At first fix $i\in\NN$. We first note that by multiplicativity of ramification indices, since the ramification types of the family $h\circ f_i$ are almost regular of genus 1, the ramification types of the family $h \circ g_i$ are regular of genus 1.  Assume $C_i$ and $D_i$ are drawn as preimages of some star tree $S\subset \mathbb{P}^1$. Once again, via topological deformation, $D_i$ can be assumed to correspond to a tiling of the plane by regular polygons, quotiented by a lattice. Let $P_i$ be a fundamental domain of this lattice.
Next, choose a star tree $S^* \subset \mathbb{P}^1$ passing through all but one branch point of  $h\circ f_i$.
Let $\tilde{C}_i$ and $\tilde{D}_i$ be the corresponding dessins for $h\circ f_i$ and $h\circ g_i$ respectively. So $\tilde{C}_i$ is the preimage $f_i^{-1}(h^{-1}(S^*))$, with the appropriate vertex coloring. Note $\tilde{D}_i$ again yields a regular tiling of the torus.

By Lemma \ref{lem:base_curve} (and as in the proof of Lemma \ref{lemma:rattranslates}), we can replace the dessins $(C_i)_i$ (resp.~$(D_i)_i$) of $(f_i)_i$ (resp.~$(g_i)_i$) by the preimages under $(f_i)_i$ (resp.\ $(g_i)_i$) of $G:=h^{-1}(S^*)$, so that $(\tilde{C}_i)_i$ are obtained by local changes to  $(\tilde{D}_i)_i$.    

It remains to show the assertion for all members of the family $(T_j)_j$ and not only those corresponding to ramification types of $(h\circ f_i)_i$. To do so, it suffices to show that the family $(\tilde{D}_i)_i$ contains disks of arbitrarily large radius, and apply Proposition \ref{prop:localch1} with sufficiently large $r$. 
To obtain this, 
we claim that the number of polygons in each neighborhood on $\mathbb{T}$ of a vertex in $\tilde{D}_i$ differs from the number of polygons in $D_i$ by at most a constant factor depending only on $h$, and independent of $i\in\NN$. Indeed, as shown in the proof of Lemma \ref{lem:base_curve}, the graph $D_i$ is obtained by gluing copies of $S$ together at the preimages of branch points in a certain way; and $\tilde{D}_i$ is obtained from $D_i$ by replacing each such copy of $S$ by a copy of $G=h^{-1}(S^*)$ and gluing together at the same preimages. Moreover, when a given copy of the star $S$ around a vertex $q_0$ on the torus is replaced by a copy of $G$, then the edges of $G$ may intersect only edges from that copy of $S$, and vice versa. In particular every side of a polygon in $D_i$ is contained in at most a constant amount $d_1$ of polygons in $\tilde D_i$, where $d_1$ is a constant depending only on $h$ and independent of $i\in\NN$.  Since both $D_i$ and $\tilde D_i$ are tilings either by hexagons or by squares, 
it follows that any disk of polygons of radius $r$ in $D_i$ contains a disk of polygons of radius at least $r/d_2$ in $\tilde{D}_i$, where $d_2$ is a constant depending only on $h$ and independent of $i\in\NN$.
 Since by assumption, $(D_i)_i$ contain disks of arbitrary large radius, this implies that $(\tilde{D}_i)_i$ contains disks of arbitrary large radius. 

Let  $c$ be an absolute constant such that  $(\tilde{C}_i)$ is obtained from $(\tilde{D}_i)_i$ by at most $c$ changes, and let $r(c)$ be the constant from Proposition \ref{prop:localch1}. Since $(\tilde D_i)_i$ contains disks of arbitrary large radius, there exists $k\in\NN$ such that $\tilde D_k$ contains a disk of radius at least $r(c)$. Thus, we can apply Proposition \ref{prop:localch1} to realize all ramification data $(T_j)_j$ of sufficiently large degree, by local changes to regular dessins which contain disks of arbitrary large radius. 
%
\end{proof}

\subsection{Proof of Theorem \ref{thm:existence}}
\label{sect:class_1}

We now come to the proof of Theorem \ref{thm:existence}. In fact, we prove the following stronger version for families of almost-regular dessins:
\begin{theorem} 
\label{thm:existence_strong} 
Let $(T_i)_{i\in\NN}$ be a family of almost-regular genus $1$ ramification data of type $[k_1,\ldots,k_r]$, error at most $\eps$, where $T_i$ is not one of four exceptional types (A)-(D) for $i\in\mathbb N$. 
Then all but finitely many members of $(T_i)_i$ are realizable if $\eps\leq 6$, or if $[k_1,\ldots,k_r]\in\{[2,2,2,2],[3,3,3]\}$ and $\eps\leq 10$. Moreover, the realizations can be chosen to be families of almost-regular dessins.
\end{theorem}
\begin{proof}
A straightforward check using magma shows that all almost-regular families in genus 1 satisfying the Riemann-Hurwitz formula  with $\varepsilon \leq 6$, and  $\varepsilon \leq 10$ for  $[3^*][3^*][3^*]$ and $[2^*][2^*][2^*][2^*]$ are as follows: 







\begin{paracol}{2} 
\begin{tabular}{ccc}
(1) & [NR] & $[3^*][3^*][3^*]$ \\
(2) & [NR] & $[3^*][3^*][1,5,3^*]$ \\
(3) & [NE] & $[3^*][3^*][2,4,3^*]$ \\
(4) & [NR] & $[1, 3^*][1, 3^*][7, 3^*]$ \\
(5) & [R1]  
&$[1, 3^*][4, 3^*][4, 3^*]$ \\
(6) & [NR] & $[2, 3^*][2, 3^*][5, 3^*]$ \\
(7) & [R4]  
& $[3^*][3^*][1,1,7,3^*]$ \\
(8) & [NR]  & $[3^*][3^*][1,2,6,3^*]$ \\
(9) & [R5]  
&$[3^*][3^*][1,4,4,3^*]$ \\
(10) & [R6]  
& $[3^*][3^*][2,2,5,3^*]$ \\
(11) & [NR]  &$[3^*][1,2,3^*][6,3^*]$ \\
& & \\
(12) & [NR]  & $[2^*][3^*][6^*]$\\
(13) & [NR]  & $[1,3,2^*] [ 3^*] [6^*]$\\
(14) & [NR]  & $[3,2^*] [ 3^*] [3,6^*]$  \\
(15) & [R2]  
&$[2^*] [1,5, 3^*] [6^*]$ \\
(16) & [NR]  &$[2^*] [ 2,4,3^*] [6^*]$ \\
(17) & [R24]  
&$[1,1,4,2^*] [ 3^*] [6^*]$ \\
& & \\
(18) & [NR]  &$[2^*][4^*][4^*]$\\
(19) & [NR]  & $[1,3,2^*][4^*] [4^*]$\\
(20) & [R24]   
& $[1,1,4,2^*][4^*] [4^*]$ 
\end{tabular}
%
%
\begin{tabular}{ccc}
(21) & [NR]  &$[2^*][2^*][2^*][2^*]$\\
(22) & [NE]  &$[2^*][2^*][2^*][1,3,2^*]$\\
(23) & [R24] 
& $[2^*][2^*][2^*][1,1,4,2^*]$\\ 
(24) & [NR]  &$[2^*][2^*][1,1,2^*][4,2^*]$\\ 
(25) & [NR]  &$[1,2^*][1,2^*][1,2^*][5,2^*]$\\ 
(26) & [R21]   
&$[1,2^*][1,2^*][3,2^*][3,2^*]$\\ 
(27) & [R26]  
&$[2^*][2^*][1,1,2^*][3,3,2^*]$\\
(28) & [R26]   
&$[2^*][2^*][1,3,2^*][1,3,2^*]$\\ 
(29) & [R25]  
&$[1,5,2^*][2^*] [ 2^*] [1^2, 2^*]$\\ 
(30) & [R29]   
&$[1^3,5,2^*][2^*][ 2^*] [2^*] $\\
(31) & [R27]   
&$[1^2,3^2,2^*][2^*] [2^*] [2^*]$\\
(32) & [R23]  
& $[1,2^*][1,2^*][1,4,2^*][3,2^*]$\\ 
(33) & [NR]  &$[2^*][1,1,2^*][1,1,2^*][6,2^*]$ \\
(34) & [R37]  
&$[2^*][2^*][1,1,1,3,2^*][4,2^*]$ \\
(35) & [R32] 
& $[2^*][2^*][1,1,4,2^*][1,3,2^*]$ \\ 
(36) & [R33]  
&$[2^*][2^*][1^4, 2^*][6,2^*]$\\
(37) & [NR]  &$[2^*][1,1,2^*][1,3,2^*][4,2^*]$\\
(38) & [R42]  
&$[1^2, 6, 2^*][2^*][2^*][1^2,2^*]$\\ 
(39) & [R36]  
&$[1^4,6,2^*][2^*][ 2^*] [2^*] $\\ 
(40) & [R37] 
&$ [1, 3,4, 2^*][2^*][ 2^*][1^2,2^*]$ \\ 
(41) & [R34] 
& $[1^3,3,4,2^*][2^*] [2^*] [2^*]$ \\ 
(42) & [NR]  &$[1,6,2^*][1,2^*][1,2^*][1,2^*]$ \\
(43) & [NR]  &$[3,4,2^*][1,2^*][1,2^*][1,2^*]$ 
\end{tabular}
\end{paracol} 

We have marked types NR for ``no reduction", R for ``reduction" followed by the number of a type that one may reduce to as justified in the proof below, and NE for ``non-existent".  Notice that NR really ought to stand for ``no reduction given," as it may be possible to reduce types in a different way than we have done here.  

 First, we show existence of all ramification types marked ``NR". 
For each ramification data we draw a dessin.  We use two essentially different approaches for drawings:\\
Most families are drawn as part of a changed regular tiling, so it clear how to extend them to larger tilings containing disks of arbitrary radius. Proposition \ref{prop:localch1} then implies the existence of almost-regular dessins of arbitrary sufficiently large degree for each ramification data with such corresponding drawing. \\
The second type of drawings is ``nested" (NR 2, 3, 13, 14, 16, and 19) . These are drawn as changes to an $n \times 1$ strip of tiles, here visualized in a fundamental domain where the top and bottom, as well as left and right sides of the drawing are identified. In this case, it is easy to see how to add one hexagon at a time. This implies that all drawings are realizable in all degrees, completing the proof for all data labeled NR.  

\begin{figure}[h!] 
\flushleft
\begin{subfigure}[b]{0.1\textwidth}
\begin{tikzpicture}[scale=.4]
\draw(-1,0) -- (0,0);
\draw(0,0)--(1,1);
\draw(1,1)--(2,1);
\draw(2,1)--(3,2);
\draw(3,2)--(2,3);
\draw(1,1)--(0,2);
\draw(0,2)--(1,3);
\draw(1,3)--(2,3);
\draw(1,3)--(0,4);
\draw(-1,4)--(0,4);
\draw(-1,4)--(-2,3);
\draw(-2,3)--(-3,3);
\draw(-3,3)--(-4,2);
\draw(-4,2)--(-3,1);
\draw(-3,1)--(-2,1);
\draw(-2,1)--(-1,2);
\draw(-1,2)--(-2,3);
\draw(-1,2)--(-0.5,1.5);
\draw(-2,1)--(-1,0);
\draw(0,2)--(-0.5,2.5);
\fill[blue] (0,0) circle (4pt);
\fill[blue] (-0.5,1.5) circle (4pt);
\fill[blue](2,1) circle (4pt);
\fill[blue](2,3) circle (4pt);
\fill[blue](0,2) circle (4pt);
\fill[blue](0,4) circle (4pt);
\fill[blue](-2,1) circle (4pt);
\fill[blue](-2,3) circle (4pt);
\fill[blue](-4,2) circle (4pt);
\fill[red](3,2) circle (4pt);
\fill[red](1,1) circle (4pt);
\fill[red](1,3) circle (4pt);
\fill[red](-1,0) circle (4pt);
\fill[red](-1,2) circle (4pt);
\fill[red](-1,4) circle (4pt);
\fill[red](-3,1) circle (4pt);
\fill[red](-3,3) circle (4pt);
\fill[red](-0.5,2.5) circle (4pt);
\end{tikzpicture} 
\caption*{NR4: $[1,3^*][1, 3^*][7 ,3^*]$}
\end{subfigure} 
\hspace{0.2\textwidth}
\begin{subfigure}[b]{0.1\textwidth}
\begin{tikzpicture}[scale=.4]
\draw(-1,0) -- (0,0);
\draw(0,0)--(1,1);
\draw(1,1)--(2,1);
\draw(2,1)--(3,2);
\draw(3,2)--(2,3);
\draw(1,1)--(0,2);
\draw(0,2)--(1,3);
\draw(1,3)--(2,3);
\draw(1,3)--(0,4);
\draw(-1,4)--(0,4);
\draw(-1,4)--(-2,3);
\draw(-2,3)--(-3,3);
\draw(-3,3)--(-4,2);
\draw(-4,2)--(-3,1);
\draw(-3,1)--(-2,1);
\draw(-2,1)--(-1,2);
\draw(-1,2)--(-2,3);
\draw(-2,1)--(-1,0);
\fill[blue] (0,0) circle (4pt);
\fill[blue](2,1) circle (4pt);
\fill[blue](2,3) circle (4pt);
\fill[blue](0,2) circle (4pt);
\fill[blue](0,4) circle (4pt);
\fill[blue](-2,1) circle (4pt);
\fill[blue](-2,3) circle (4pt);
\fill[blue](-4,2) circle (4pt);
\fill[red](3,2) circle (4pt);
\fill[red](1,1) circle (4pt);
\fill[red](1,3) circle (4pt);
\fill[red](-1,0) circle (4pt);
\fill[red](-1,2) circle (4pt);
\fill[red](-1,4) circle (4pt);
\fill[red](-3,1) circle (4pt);
\fill[red](-3,3) circle (4pt);
\end{tikzpicture} 
\caption*{NR6: $[2 ,3^*][2, 3^*][5 ,3^*]$}
\end{subfigure} 
\hspace{0.2\textwidth}
\begin{subfigure}[b]{0.1\textwidth}
\begin{tikzpicture}[scale=.4]
\draw(-1,0) -- (0,0);
\draw(3,0)--(2,1);
\draw(1,1)--(2,1);
\draw(2,1)--(3,2);
\draw(3,2)--(2,3);
\draw(1,1)--(0,2);
\draw(0,2)--(1,3);
\draw(1,3)--(2,3);
\draw(1,3)--(0,4);
\draw(-1,4)--(0,4);
\draw(-1,4)--(-2,3);
\draw(-2,3)--(-3,3);
\draw(-3,3)--(-4,2);
\draw(-4,2)--(-3,1);
\draw(-3,1)--(-2,1);
\draw(-2,1)--(-1,2);
\draw(-1,2)--(-2,3);
\draw(-1,2)--(0,0);
\draw(-1,0)--(-2,1);
\draw(3,0)--(2,-1);
\draw(1,-1)--(2,-1);
\draw(0,0)--(1,-1);
\draw(0,2)to[bend right](1,1);
\fill[blue] (0,0) circle (4pt);
\fill[blue](2,1) circle (4pt);
\fill[blue](2,3) circle (4pt);
\fill[blue](0,2) circle (4pt);
\fill[blue](0,4) circle (4pt);
\fill[blue](-2,1) circle (4pt);
\fill[blue](-2,3) circle (4pt);
\fill[blue](-4,2) circle (4pt);
\fill[red](3,2) circle (4pt);
\fill[red](1,1) circle (4pt);
\fill[red](1,3) circle (4pt);
\fill[red](-1,0) circle (4pt);
\fill[red](-1,2) circle (4pt);
\fill[red](-1,4) circle (4pt);
\fill[red](-3,1) circle (4pt);
\fill[red](-3,3) circle (4pt);
\fill[red](1,-1) circle (4pt);
\fill[red](3,0) circle (4pt);
\fill[blue](2,-1) circle (4pt);
\end{tikzpicture} 
\caption*{NR8: $[3^*][3^*][1,2,6, 3^*]$}
\end{subfigure}
\end{figure} 

\begin{figure}[h!] 
\flushleft
\begin{subfigure}[b]{0.1\textwidth}
\begin{tikzpicture}[scale=.4]
\draw(-1,0) -- (0,0);
\draw(0,0)--(1,1);
\draw(1,1)--(2,1);
\draw(2,1)--(3,2);
\draw(3,2)--(2,3);
\draw(1,1)--(0,2);
\draw(0,2)--(1,3);
\draw(1,3)--(2,3);
\draw(1,3)--(0,4);
\draw(-1,4)--(0,4);
\draw(-1,4)--(-2,3);
\draw(-2,3)--(-3,3);
\draw(-3,3)--(-4,2);
\draw(-4,2)--(-3,1);
\draw(-3,1)--(-2,1);
\draw(-2,1)--(-1,2);
\draw(-1,2)--(-2,3);
\draw(-2,1)--(-1,0);
\draw(0,2)--(-0.5,2.5);
\fill[blue] (0,0) circle (4pt);
\fill[blue](2,1) circle (4pt);
\fill[blue](2,3) circle (4pt);
\fill[blue](0,2) circle (4pt);
\fill[blue](0,4) circle (4pt);
\fill[blue](-2,1) circle (4pt);
\fill[blue](-2,3) circle (4pt);
\fill[blue](-4,2) circle (4pt);
\fill[red](3,2) circle (4pt);
\fill[red](1,1) circle (4pt);
\fill[red](1,3) circle (4pt);
\fill[red](-1,0) circle (4pt);
\fill[red](-1,2) circle (4pt);
\fill[red](-1,4) circle (4pt);
\fill[red](-3,1) circle (4pt);
\fill[red](-3,3) circle (4pt);
\fill[red](-0.5,2.5) circle (4pt);
\end{tikzpicture} 
\caption*{NR11: $[3^*][1,2, 3^*][6, 3^*]$}
\end{subfigure} 
\hspace{0.2\textwidth}
\begin{subfigure}[b]{0.1\textwidth}
\begin{tikzpicture}[scale=.3]
\draw(2,0) -- (0,0);
\draw(2,0)--(4,2);
\draw(4,2)--(2,4);
\draw(2,4)--(4,6);
\draw(4,6)--(2,8);
\draw(2,8)--(0,8);
\draw(0,8)--(-2,6);
\draw(-2,6)--(-4,6);
\draw(-4,6)--(-6,4);
\draw(-6,4)--(-4,2);
\draw(-4,2)--(-2,2);
\draw(-2,2)--(0,0);
\draw(-2,6)--(0,4);
\draw(-2,2)--(0,4);
\draw(4,6)--(6,6);
\draw(6,6)--(8,4);
\draw(8,4)--(6,2);
\draw(6,2)--(4,2);
\draw(2,4)--(1,5);
\draw(0,4)--(1,3);
\fill[blue](-1,1) circle (4pt);
\fill[blue](3,7) circle (4pt);
\fill[blue](3,3) circle (4pt);
\fill[blue](-5,3) circle (4pt);
\fill[blue](-1,5) circle (4pt);
\fill[blue](7,5) circle (4pt);
\fill[red](-3,2) circle (4pt);
\fill[red](-3,6) circle (4pt);
\fill[red](1,8) circle (4pt);
\fill[red](1,0) circle (4pt);
\fill[red](5,6) circle (4pt);
\fill[red](5,2) circle (4pt);
\fill[red](1,5) circle (4pt);
\fill[red](1,3) circle (4pt);
\fill[green](-5,5) circle (4pt);
\fill[green](-1,3) circle (4pt);
\fill[green](-1,7) circle (4pt);
\fill[green](3,5) circle (4pt);
\fill[green](3,1) circle (4pt);
\fill[green](7,3) circle (4pt);
\end{tikzpicture} 
\caption*{NR24: $[2^*][2^*][1,1, 2^*][4 ,2^*]$}
\end{subfigure} 
\hspace{0.2\textwidth}
\begin{subfigure}[b]{0.1\textwidth}
\begin{tikzpicture}[scale=.3]
\draw(2,0) -- (0,0);
\draw(2,0)--(4,2);
\draw(4,2)--(2,4);
\draw(2,4)--(4,6);
\draw(4,6)--(2,8);
\draw(2,8)--(0,8);
\draw(0,8)--(-2,6);
\draw(-2,6)--(-4,6);
\draw(-4,6)--(-6,4);
\draw(-6,4)--(-4,2);
\draw(-4,2)--(-2,2);
\draw(-2,2)--(0,0);
\draw(-2,6)--(-1,5);
\draw(2,4)--(0,4);
\draw(-2,2)--(-1,3);
\fill[blue] (-5,3) circle (4pt);
\fill[blue] (-1,5) circle (4pt);
\fill[blue](-1,1) circle (4pt);
\fill[blue](3,7) circle (4pt);
\fill[blue](3,3) circle (4pt);
\fill[red](-3,2) circle (4pt);
\fill[red](-3,6) circle (4pt);
\fill[red](1,8) circle (4pt);
\fill[red](0,4) circle (4pt);
\fill[red](1,0) circle (4pt);
\fill[green](-5,5) circle (4pt);
\fill[green](-1,3) circle (4pt);
\fill[green](-1,7) circle (4pt);
\fill[green](3,5) circle (4pt);
\fill[green](3,1) circle (4pt);
\end{tikzpicture} 
\caption*{NR25: $[1 ,2^*][1 ,2^*][1 ,2^*][5, 2^*]$}
\end{subfigure} 

\end{figure} 

\begin{figure}[h!] 
\flushleft
\begin{subfigure}[b]{0.1\textwidth}
\begin{tikzpicture}[scale=.3]
\draw(2,0) -- (0,0);
\draw(2,0)--(4,2);
\draw(4,2)--(3,3);
\draw(3,5)--(4,6);
\draw(4,6)--(2,8);
\draw(2,8)--(0,8);
\draw(0,8)--(-2,6);
\draw(-2,6)--(-4,6);
\draw(-4,6)--(-6,4);
\draw(-6,4)--(-4,2);
\draw(-4,2)--(-2,2);
\draw(-2,2)--(0,0);
\draw(-2,6)--(-1,5);
\draw(-2,2)--(-1,3);
\draw(4,6)--(6,6);
\draw(6,6)--(8,4);
\draw(8,4)--(6,2);
\draw(6,2)--(4,2);
\fill[blue](-1,1) circle (4pt);
\fill[blue](3,7) circle (4pt);
\fill[blue](3,3) circle (4pt);
\fill[blue](-5,3) circle (4pt);
\fill[blue](-1,5) circle (4pt);
\fill[blue](7,5) circle (4pt);
\fill[red](-3,2) circle (4pt);
\fill[red](-3,6) circle (4pt);
\fill[red](1,8) circle (4pt);
\fill[red](1,0) circle (4pt);
\fill[red](5,6) circle (4pt);
\fill[red](5,2) circle (4pt);
\fill[green](-5,5) circle (4pt);
\fill[green](-1,3) circle (4pt);
\fill[green](-1,7) circle (4pt);
\fill[green](3,5) circle (4pt);
\fill[green](3,1) circle (4pt);
\fill[green](7,3) circle (4pt);
\end{tikzpicture} 
\caption*{NR33: $[2^*][11 2^*][11 2^*][6 2^*]$}
\end{subfigure} 
\hspace{0.4\textwidth}
\begin{subfigure}[b]{0.1\textwidth}
\begin{tikzpicture}[scale=.3]
\draw(2,0) -- (0,0);
\draw(2,0)--(4,2);
\draw(4,2)--(2,4);
\draw(4,6)--(2,8);
\draw(0,8)--(-2,6);
\draw(-2,6)--(-4,6);
\draw(-4,6)--(-6,4);
\draw(-6,4)--(-4,2);
\draw(-4,2)--(-2,2);
\draw(-2,2)--(0,0);
\draw(-2,6)--(0,4);
\draw(-2,2)--(0,4);
\draw(4,6)--(6,6);
\draw(6,6)--(8,4);
\draw(8,4)--(6,2);
\draw(6,2)--(4,2);
\draw(6,6)--(8,8);
\draw(8,8)--(6,10);
\draw(6,10)--(4,10);
\draw(4,10)--(2,8);
\draw(4,10)--(2,12);
\draw(2,12)--(0,12);
\draw(0,12)--(-2,10);
\draw(-2,10)--(-4,10);
\draw(-2,10)--(0,8);
\draw(-4,10)--(-6,8);
\draw(-6,8)--(-4,6);
\draw(2,8)--(1.5,9);
\draw(2,4)to[bend right](-1,11);
\draw(0,8)--(0.5,9);
\draw(2,4)--(0,4);
\draw(4,6)--(4.5,5);
\fill[blue](-1,1) circle (4pt);
\fill[blue](3,7) circle (4pt);
\fill[blue](3,3) circle (4pt);
\fill[blue](-5,3) circle (4pt);
\fill[blue](-1,5) circle (4pt);
\fill[blue](7,5) circle (4pt);
\fill[blue](-5,7) circle (4pt);
\fill[blue](-1,9) circle (4pt);
\fill[blue](3,11) circle (4pt);
\fill[blue](7,9) circle (4pt);
\fill[red](-3,2) circle (4pt);
\fill[red](-3,6) circle (4pt);
\fill[red](1,0) circle (4pt);
\fill[red](1,4) circle (4pt);
\fill[red](5,6) circle (4pt);
\fill[red](5,2) circle (4pt);
\fill[red](1.5,9) circle (4pt);
\fill[red](-3,10) circle (4pt);
\fill[red](1,12) circle (4pt);
\fill[red](5,10) circle (4pt);
\fill[red](0.5,9) circle (4pt);
\fill[green](-1,3) circle (4pt);
\fill[green](4.5,5) circle (4pt);
\fill[green](-1,7) circle (4pt);
\fill[green](3,1) circle (4pt);
\fill[green](7,3) circle (4pt);
\fill[green](-5,9) circle (4pt);
\fill[green](-1,11) circle (4pt);
\fill[green](3,9) circle (4pt);
\fill[green](7,7) circle (4pt);
\fill[green](-5,5) circle (4pt);
\end{tikzpicture} 
\caption*{NR37: $[2^*][1,1, 2^*][1,3, 2^*][4, 2^*]$}
\end{subfigure} 
\end{figure}

\begin{figure}[h!]  
\flushleft
\begin{subfigure}[b]{0.1\textwidth}
\begin{tikzpicture}[scale=.3]
\draw(2,0) -- (0,0);
\draw(2,0)--(4,2);
\draw(4,2)--(2,4);
\draw(2,4)--(4,6);
\draw(4,6)--(2,8);
\draw(-1,7)--(-2,6);
\draw(-2,6)--(-4,6);
\draw(-4,6)--(-6,4);
\draw(-6,4)--(-4,2);
\draw(-4,2)--(-2,2);
\draw(-2,2)--(0,0);
\draw(-2,6)--(0,4);
\draw(-2,2)--(0,4);
\draw(4,6)--(6,6);
\draw(6,6)--(8,4);
\draw(8,4)--(6,2);
\draw(6,2)--(4,2);
\draw(6,6)--(8,8);
\draw(8,8)--(6,10);
\draw(6,10)--(4,10);
\draw(4,10)--(2,8);
\draw(4,10)--(2,12);
\draw(2,12)--(0,12);
\draw(0,12)--(-2,10);
\draw(-2,10)--(-4,10);
\draw(-2,10)--(-1,9);
\draw(-4,10)--(-6,8);
\draw(-6,8)--(-4,6);
\draw(2,8)--(2,4);
\draw(0,4)--(1,3);
\fill[blue](-1,1) circle (4pt);
\fill[blue](3,7) circle (4pt);
\fill[blue](3,3) circle (4pt);
\fill[blue](-5,3) circle (4pt);
\fill[blue](-1,5) circle (4pt);
\fill[blue](7,5) circle (4pt);
\fill[blue](-5,7) circle (4pt);
\fill[blue](-1,9) circle (4pt);
\fill[blue](3,11) circle (4pt);
\fill[blue](7,9) circle (4pt);
\fill[red](-3,2) circle (4pt);
\fill[red](2,6) circle (4pt);
\fill[red](-3,6) circle (4pt);
\fill[red](1,0) circle (4pt);
\fill[red](5,6) circle (4pt);
\fill[red](5,2) circle (4pt);
\fill[red](1,3) circle (4pt);
\fill[red](-3,10) circle (4pt);
\fill[red](1,12) circle (4pt);
\fill[red](5,10) circle (4pt);
\fill[green](-1,3) circle (4pt);
\fill[green](-1,7) circle (4pt);
\fill[green](3,5) circle (4pt);
\fill[green](3,1) circle (4pt);
\fill[green](7,3) circle (4pt);
\fill[green](-5,9) circle (4pt);
\fill[green](-1,11) circle (4pt);
\fill[green](3,9) circle (4pt);
\fill[green](7,7) circle (4pt);
\fill[green](-5,5) circle (4pt);
\end{tikzpicture} 
\caption*{NR42: $[1,6, 2^*][1,2^*]^{3}$}
\end{subfigure} 
\hspace{0.4\textwidth}
\begin{subfigure}[b]{0.1\textwidth}
\begin{tikzpicture}[scale=.3]
\draw(2,0) -- (0,0);
\draw(2,0)--(4,2);
\draw(4,2)--(2,4);
\draw(2,4)--(4,6);
\draw(4,6)--(2,8);
\draw(-3,6)--(-4,6);
\draw(-4,6)--(-6,4);
\draw(-6,4)--(-4,2);
\draw(-4,2)--(-2,2);
\draw(-2,2)--(0,0);
\draw(4,10)--(2,8);
\draw(4,10)--(2,12);
\draw(2,12)--(0,12);
\draw(0,12)--(-2,10);
\draw(-2,10)--(-4,10);
\draw(-2,10)--(0,8);
\draw(-4,10)--(-6,8);
\draw(-6,8)--(-4,6);
\draw(2,8)--(0,8);
\draw(0,0)--(-2,-2);
\draw(-2,-2)--(-4,-2);
\draw(-4,-2)--(-6,0);
\draw(-2,2)--(0,8);
\draw(-6,0)--(-4,2);
\draw(2,4)--(0.5,4);
\draw(1,5)--(0.5,4);
\draw(0,3)--(0.5,4);
\fill[blue](-1,1) circle (4pt);
\fill[blue](1,5) circle (4pt);
\fill[blue](3,7) circle (4pt);
\fill[blue](3,3) circle (4pt);
\fill[blue](-5,3) circle (4pt);
\fill[blue](-5,7) circle (4pt);
\fill[blue](-1,9) circle (4pt);
\fill[blue](3,11) circle (4pt);
\fill[blue](-5,-1) circle (4pt);
\fill[red](-3,2) circle (4pt);
\fill[red](-3,6) circle (4pt);
\fill[red](1,0) circle (4pt);
\fill[red](1,4) circle (4pt);
\fill[red](1,8) circle (4pt);
\fill[red](-3,-2) circle (4pt);
\fill[red](-3,10) circle (4pt);
\fill[red](1,12) circle (4pt);
\fill[green](-1,5) circle (4pt);
\fill[green](0,3) circle (4pt);
\fill[green](-5,1) circle (4pt);
\fill[green](-1,-1) circle (4pt);
\fill[green](3,5) circle (4pt);
\fill[green](3,1) circle (4pt);
\fill[green](-5,9) circle (4pt);
\fill[green](-1,11) circle (4pt);
\fill[green](3,9) circle (4pt);
\fill[green](-5,5) circle (4pt);
\end{tikzpicture} 
\caption*{NR43: $[3,4, 2^*][1,2^*]^{3}$}
\end{subfigure} 
\end{figure}

\begin{figure}[h!] 
\begin{subfigure}[b]{0.1\textwidth}
\begin{tikzpicture}[scale=.4]
\draw(0,6)--(0,0);
\draw(0,7)--(0,9);
\draw(-1,2)--(2,2);
\draw(-1,4)--(2,4);
\draw(-1,6)--(2,6);
\draw(-1,8)--(2,8);
\draw(1,8)to[bend left](0,6);
\fill[blue](1,8) circle (4pt);
\fill[blue](0,7) circle (4pt);
\fill[blue](1,6) circle (4pt);
\fill[blue](0,5) circle (4pt);
\fill[blue](1,4) circle (4pt);
\fill[blue](0,1) circle (4pt);
\fill[blue](1,2) circle (4pt);
\fill[blue](0,3) circle (4pt);
\fill[red](0,8) circle (4pt);
\fill[red](0,6) circle (4pt);
\fill[red](0,4) circle (4pt);
\fill[red](0,2) circle (4pt);
\end{tikzpicture} 
\caption*{NR19: $[1,3, 2^*][4^*][4^*]$}
\end{subfigure}
\hspace{0.2\textwidth}
\begin{subfigure}[b]{0.1\textwidth}
\begin{tikzpicture}[scale=.4]
\draw(-2,-1)--(0,1);
\draw(2,-1)--(-1,2);
\draw(-3,2)--(0,5);
\draw(2,3)--(-2,7);
\draw(-3,6)--(0,9);
\draw(-1,10)--(2,7);
\draw(-2,3)to[bend left](1,0);
\fill[blue](1,8) circle (4pt);
\fill[blue](-1,8) circle (4pt);
\fill[blue](-1,6) circle (4pt);
\fill[blue](-1,4) circle (4pt);
\fill[blue](1,4) circle (4pt);
\fill[blue](-1,2) circle (4pt);
\fill[blue](1,0) circle (4pt);
\fill[red](0,9) circle (4pt);
\fill[red](-2,7) circle (4pt);
\fill[red](0,5) circle (4pt);
\fill[red](0,1) circle (4pt);
\fill[red](-2,3) circle (4pt);
\end{tikzpicture} 
\caption*{NR13: $[1,3 ,2^*][3^*][6^*]$}
\end{subfigure}
\hspace{0.2\textwidth}
\begin{subfigure}[b]{0.1\textwidth}
\begin{tikzpicture}[scale=.4]
\draw(-2,-1)--(0,1);
\draw(-2,3)--(2,-1);
\draw (-2, 3) -- ( -3, 2); 
\draw(-1,4)--(0,5);
\draw (2,3)-- (0, 5); 
\draw(-3,6)--(0,9);
\draw(-1,10)--(2,7);
\draw (-2,7) -- (-2,3) ; 
\draw (-1, 4) to (-2,3); 
\fill[blue](1,8) circle (4pt);
\fill[blue](-1,8) circle (4pt);
\fill[blue](-1,4) circle (4pt);
\fill[blue](1,4) circle (4pt);
\fill[blue](-1,2) circle (4pt);
\fill[blue](1,0) circle (4pt);
\fill[blue](-2,5) circle (4pt);
\fill[red](0,9) circle (4pt);
\fill[red](-2,7) circle (4pt);
\fill[red](0,5) circle (4pt);
\fill[red](0,1) circle (4pt);
\fill[red](-2,3) circle (4pt);
\end{tikzpicture} 
\caption*{NR16: $[2^*][2,4, 3^*][6^*]$}
\end{subfigure}
\end{figure}

\begin{figure} 

\begin{subfigure}[b]{0.1\textwidth}
\begin{tikzpicture}[scale=.4]
\draw(-2,-1)--(0,1);
\draw(-2,3)--(2,-1);
\draw(-2,3)--(-3,2);
\draw(-2,3)--(0,5);
\draw(2,3)--(-2,7);
\draw(-3,6)--(0,9);
\draw (-1, 0 ) -- (-1, -2); 
\draw ( -1, -2) -- (-2, -3) ; 
\draw ( -1, -2) -- ( 1,0); 
\fill[blue](-1,8) circle (4pt);
\fill[blue](-1,6) circle (4pt);
\fill[blue](-1,4) circle (4pt);
\fill[blue](1,4) circle (4pt);
\fill[blue](-1,2) circle (4pt);
\fill[blue](1,0) circle (4pt);
\fill[blue](-1,-1) circle (4pt);
\fill[blue](-.5,.5) circle (4pt);
\fill[red](-2,7) circle (4pt);
\fill[red](0,5) circle (4pt);
\fill[red](0,1) circle (4pt);
\fill[red](-1,0) circle (4pt);
\fill[red](-1,-2) circle (4pt);
\fill[red](-2,3) circle (4pt);
\end{tikzpicture} 
\caption*{NR14: $[3 ,2^*][3^*][3 ,6^*]$}
\end{subfigure}
\hspace{0.4\textwidth}
\begin{subfigure}[b]{0.1\textwidth}
\begin{tikzpicture}[scale=.3]
\draw(-1,10)--(2,7);
\draw(-2,3)--(0,1);
\draw(-3,2)--(0,5);
\draw(1,4)--(-2,7);
\draw(-3,6)--(0,9);
\draw(3,-6)--(0,-3);
\draw(0,-3)--(-3,-6);
\draw(-2,-5)--(-1,-6);
\draw(0,1)to[bend right](0,-1);
\draw(0,1)to[bend left](0,-1);
\draw(0,-1)--(0,-3);
\fill[blue](-2,7) circle (4pt);
\fill[blue](-2,3) circle (4pt);
\fill[blue](0,-1) circle (4pt);
\fill[blue](-2,-5) circle (4pt);
\fill[red](0,9) circle (4pt);
\fill[red](0,5) circle (4pt);
\fill[red](0,1) circle (4pt);
\fill[red](0,-3) circle (4pt);
\end{tikzpicture} 
\caption*{NR2: $[3^*][3^*][1,5, 3^*]$}
\end{subfigure}

\end{figure} 

\newpage

Non-existence of types marked NE is shown is section \ref{nonex:genus1}.

We next show the existence of all the types marked R  by reducing to types marked NR, thereby completing the proof of Theorem \ref{thm:existence_strong}. 
The data 5,26,32 is reduced by Lemma \ref{lemma:addlines} applied with $k=1$. 
For the rest of the cases marked ``R" we use Lemma \ref{lemma:rattranslates}. 
Case 15 is reduced to case 2 as follows. 
Let $g$ be the genus-$0$ covering of the sphere given by $x\mapsto x^2$, and let $f$ be a covering with ramification type $[1,5,3^*][3^*][3^*]$ and branch points (e.g.) $1$, $-1$, $\infty$ (in this order). Then  the composition $g\circ f$ has branch points $0$, $1$, $\infty$ and ramification type $[2^*][1,5,3^*][6^*]$. Since the first type has already been realized in all degrees divisible by $3$, this argument yields realizability of the second type in all degrees divisible by $6$, which is all possible degrees.

For all other cases in which we apply Lemma \ref{lemma:rattranslates}, we reduce to families of dessins which contain disks of arbitrary large radius. 
Existence for types with at least two occurrences of the partition $[2^*]$ are reduced as follows. 
Let $g$ be the genus-$0$ covering of the sphere given by $x\mapsto x^2$, and let $f$ be a covering with ramification type that we reduce to, and assume its branch points are $1$, $-1$, $\sqrt{-1}$, $-\sqrt{-1}$ in this order. Then  the composition $g\circ f$ has two branch points, $0$ and $\infty$, with all ramification indices equal to $2$, whereas two branch points the two branch points $1,-1$ (resp.~$\sqrt{-1}$, $-\sqrt{-1}$) of $f$ merge into one branch point of $g\circ f$. By Lemma \ref{lem:glue2} the reduced ramification data are realizable in all sufficiently large degrees. 
Similarly,  the other reductions via Lemma \ref{lemma:rattranslates} are given in the above tables are obtained by composing families of covering $f$ realized by dessins containing disks of arbitrary radius,  with the genus-$0$ covering $g$ given either by $x\mapsto x^2$ or $x\mapsto x^3$. 
%
\end{proof} 



\subsection{Almost-regular families with bounded ramification indices}
The next lemma shows realizations by almost-regular dessins can be combined.  It is an ingredient in the proof of Theorem \ref{123exist}, which is an existence argument for almost-regular genus $1$ types whose ramification indices are all bounded from above by $3$.
\begin{lemma}\label{lem:comb}
Suppose  the families of ramification data $[A_j,e_j^*]$, $[B_j,e_j^*]$, $j=1,\ldots,r$, are realizable by almost-regular dessins containing disks of arbitrarily large radius. Then the family $[A_j,B_j,e_j^*]$, $j=1,\ldots,r$ is realizable by almost-regular dessins containing disks of arbitrary large radius. 
\end{lemma}
\begin{proof}
The same proof as in Proposition \ref{prop:localch1} applies, only here in sufficiently large degrees, both the changes for $T_1$ and $T_2$ can be applied simultaneously in a non-overlapping manner to a regular torus tiling, which then results in the type $T$.
\end{proof}

\begin{proof}[Proof of Theorem \ref{123exist}:]
Below, we show existence in {\it infinitely many} degrees. Since all families of dessins of type [2,2,2,2] constructed   above, either by drawing or by reduction arguments,  contain disks of arbitrary radius, Proposition \ref{prop:localch1} implies the existence of almost-regular dessins in {\it all} sufficiently large degrees. 

Let $T:=[1^{k_1},3^{m_1},2^*][1^{k_2},3^{m_2},2^*][1^{k_3},3^{m_3},2^*][1^{k_4},3^{m_4},2^*]$ be a family of almost-regular ramification data . Since such families with error at most $10$ are realizable by Theorem \ref{thm:existence}, we can assume that $T$ is almost-regular with error $>10$. 
The idea of the following proof is to ``split" the type $T$ into two halves, in the following way. Set
$$T_1:=[1^{a_1},3^{b_1},2^*][1^{a_2},3^{b_2},2^*][1^{a_3},3^{b_3},2^*][1^{a_4},3^{b_4},2^*],$$ 
$$T_2:=[1^{c_1},3^{d_1},2^*][1^{c_2},3^{d_2},2^*][1^{c_3},3^{d_3},2^*][1^{c_4},3^{d_4},2^*],$$ with integers $a_i,b_i,c_i,d_i$ such that 
$a_i+c_i=k_i$ and $b_i+d_i=m_i$ for each $i=1,\ldots,4$ and such that $\sum_{i=1}^4 a_i = \sum_{i=1}^4 b_i$. The last equality ensures that $T_1$ and $T_2$ are again of genus $1$. 

 If we can now show that the families $T_1$ and $T_2$ are realizable by almost-regular dessins containing disks of arbitrary large radius, then by Lemma \ref{lem:comb} the assertion follows for the type $T$. 

Now if we can choose the two types $T_1$ and $T_2$ such that: 
\begin{equation}\label{equ:dec} 
\text{$T_1$ and $T_2$ are different from  $[2^*]^4$ and $[1,3,2^*][2^*]^3$},
\end{equation}
then the assertion follows by induction on the error. 

If $T$ is of even degree, so that $k_i+m_i$ is even for all $i$, then up to permuting the partitions of $T$, we can assume one of the following:
\begin{itemize}
\item[(a)] $k_1,m_1\ge 2$.
\item[(b)] $k_1\le 1$ and $m_1\ge 2$.
\item[(c)] $m_i\le 1$ for all $i=1,...,4$.
\end{itemize}

We treat these cases separately.
In Case (a), define $T_1$ and $T_2$ via $a_1=b_1=2$, $a_i=b_i=0$ for all $i\in \{2,3,4\}$, thus forcing $c_1=k_1-2$, $d_1=m_1-2$, $c_i=k_i$ and $d_i=m_i$ for all $i\in \{2,3,4\}$. Assuming the error is $>10$, this leads to data $T_1$, $T_2$ as desired in \eqref{equ:dec} unless or $T$ is one of 1) $[1^3,3^3,2^*][2^*]^3$ or 2) $[1^2,3^2,2^*][1,3,2^*][2^*]^2$; in those two cases, $T_2$ would become a non-existent type.

In Case (b), because of even degree and genus $1$, we have $k_i\ge 2$ for some $i\in \{2,3,4\}$, without loss for $i=2$. Define $T_1$ and $T_2$ by $b_1=2=a_2$ and $a_i,b_j=0$ otherwise. This works unless $T$ was one of 3) $[1,3^3,2^*][1^2,2^*][2^*]^2$, 4) $[3^2,2^*][1^3,3,2^*][2^*]^2$, or 5) $[3^2,2^*][1^2,2^*][1,3,2^*][2^*]$. 

In Case (c), note that the condition $\sum_{i=1}^4 k_i = \sum_{i=1}^4 m_i$ together with even degree leave only two types with error $>10$, namely $[1,3,2^*]^3[2^*]$ and $[1,3,2^*]^4$. The latter one can be split into $[1,3,2^*]^2[2^*]^2$ and $[2^*]^2[1,3,2^*]^2$. The type 6) $[1,3,2^*]^3[2^*]$ cannot be split, and is therefore another exception.

\vspace{2mm}

Next, assume $T$ is of odd degree, so $k_i+m_i$ is odd for all $i$ and in particular, every partition in $T$ must contain at least one odd entry. Then, up to permuting the partitions in $T$, one of the following holds:
\begin{itemize}
\item[(d)] $k_1,k_2\ge 1$ and $m_3,m_4\ge 1$.
\item[(e)] $m_2=m_3=m_4=0$.
\item[(f)] $k_2=k_3=k_4=0$.
\end{itemize}

In Case (d), set $a_1=a_2=b_3=b_4=1$ and $a_i,b_j=0$ otherwise. This works except when $T$ is one of 7) $[1^2,3,2^*][1,2^*][3,2^*]^2$ and 8) $[1,2^*]^2[3,2^*][1,3^2,2^*]$.

In Case (e), note that $k_2,k_3,k_4\ge 1$, and hence $m_1\ge 3$. Set $a_2=a_3=a_4=1$, $b_1=3$ and $a_i,b_j=0$ otherwise. This succeeds unless $T$ is one of 9) $[1,2^*]^3[3^3,2^*]$ or 10) $[1,2^*]^3[1,3^4,2^*]$. 

Finally, Case (f) can be treated just like Case (e), with the roles of entries equal to $1$ and equal to $3$ exchanged. The exceptional cases here are then exactly 11) $[3,2^*]^3[1^3,2^*]$ and 12) $[3,2^*]^3[1^4,3,2^*]$.

\vspace{2mm}

It therefore suffices to deal with the total of $12$ above ``exceptional" types. 
Firstly, the six types 5) $[3^2,2^*][1^2,2^*][1,3,2^*][2^*]$, 6) $[1,3,2^*]^3[2^*]$, 9) $[1,2^*]^3[3^3,2^*]$, 10) $[1,2^*]^3[1,3^4,2^*]$, 11) $[3,2^*]^3[1^3,2^*]$, 12) $[3,2^*]^3[1^4,3,2^*]$,   are drawn in at the end of the proof. 
The following types (on the right) arise via appropriate composition of genus zero and genus one types by Lemma \ref{lemma:rattranslates}:
\begin{itemize}
\item[5)$\to$ 3):] $[3^2,2^*][1^2,2^*][1,3,2^*][2^*] \to [1,3^3,2^*][1^2,2^*][2^*]^2$;
\item[3)$\to$ 1):] $[1,3^3,2^*][1^2,2^*][2^*]^2 \to [1^3,3^3,2^*][2^*]^3$;
\item[6)$\to$ 2):] $[1,3,2^*]^3[2^*]\to [1^2,3^2,2^*][1,3,2^*][2^*]^2$;
\item[11)$\to$ 4):] $[3,2^*]^3[1^3,2^*]\to [3^2,2^*][1^3,3,2^*][2^*]^2$.
\end{itemize}
In the first reduction 5)$\to$3) we compose a genus-$1$ covering $f$ with the ramification type on the left with a degree-$2$ genus-$0$ covering $g$ of ramification type $[2][2]$, with the following conditions:\\
The two points of ramification structure $[1^2,2^*]$ and $[2^*]$ for $f$ lie in a common fiber of $g$; same for the two points of ramification structure $[3^2,2^*]$ and $[1,3,2^*]$; the two ramified points of $g$ do not ramify further under $f$.\\
For 3)$\to$1), we compose the genus zero ramification type $[2][2]$ in the obvious way with the genus one type on the left. A similar composition works in the other two reductions. \\
%
\vspace{2mm}
 Note that for these four reductions, we make use of Proposition \ref{prop:localch1} and Lemma \ref{lem:glue2}. To apply these, it suffices to note that the genus-$1$ families were realized by almost-regular dessins  containing disks of arbitrary radius.

Types 7)-8) arise from the types on the left via Lemma \ref{lemma:addlines} with $k=1$:
\begin{itemize}
\item[] $[{\textcolor{red}1},1,2^*][{\textcolor{red}2},2^*][2^*][3^2,2^*] \to [1,2^*]^2[3,2^*][1,3^2,2^*]$;
\item[] $[1,3,2^*][2^*][{\textcolor{red}2},2^*][{\textcolor{red}1},3,2^*] \to [1^2,3,2^*][1,2^*][3,2^*]^2$.
\end{itemize}

The dessins that follow complete the proof for the 12 exceptional types and hence the theorem. 
\end{proof}

\begin{figure}[h!] 
\flushleft
\begin{subfigure}[b]{0.1\textwidth}
\begin{tikzpicture}[scale=.3]
\draw(2,0)--(4,2);
\draw(4,2)--(3,3);
\draw(3,5)--(4,6);
\draw(4,6)--(2,8);
\draw(0,8)--(-2,6);
\draw(-2,6)--(-4,6);
\draw(-4,6)--(-6,4);
\draw(-6,4)--(-4,2);
\draw(-4,2)--(-3,2);
\draw(-2,6)--(0,4);
\draw(-1,3)--(0,4);
\draw(4,6)--(6,6);
\draw(6,6)--(8,4);
\draw(8,4)--(6,2);
\draw(6,2)--(4,2);
\draw(6,6)--(8,8);
\draw(8,8)--(6,10);
\draw(6,10)--(4,10);
\draw(4,10)--(2,8);
\draw(4,10)--(2,12);
\draw(2,12)--(0,12);
\draw(0,12)--(-2,10);
\draw(-2,10)--(-4,10);
\draw(-2,10)--(0,8);
\draw(-4,10)--(-6,8);
\draw(-6,8)--(-4,6);
\draw(6,2)--(8,0);
\draw(8,0)--(6,-2);
\draw(6,-2)--(4,-2);
\draw(4,-2)--(2,-4);
\draw(2,-4)--(0,-4);
\draw(0,-4)--(-2,-2);
\draw(-2,-2)--(0,0);
\draw(-2,-2)--(-4,-2);
\draw(-4,-2)--(-6,0);
\draw(-6,0)--(-4,2);
\draw(2,0)--(4,-2);
\draw(2,0)--(0,0);
\draw(0,8)--(2,8);
\draw(3,1)--(-1,3);
\draw(-5,3)to[bend left](0,0);
\draw(0,4)to[bend right](5,6);
\fill[blue](3,7) circle (4pt);
\fill[blue](3,3) circle (4pt);
\fill[blue](-5,3) circle (4pt);
\fill[blue](-1,5) circle (4pt);
\fill[blue](7,5) circle (4pt);
\fill[blue](-5,7) circle (4pt);
\fill[blue](-1,9) circle (4pt);
\fill[blue](3,11) circle (4pt);
\fill[blue](7,9) circle (4pt);
\fill[blue](7,1) circle (4pt);
\fill[blue](3,-1) circle (4pt);
\fill[blue](-1,-3) circle (4pt);
\fill[blue](-5,-1) circle (4pt);
\fill[red](-3,2) circle (4pt);
\fill[red](-3,6) circle (4pt);
\fill[red](1,8) circle (4pt);
\fill[red](1,0) circle (4pt);
\fill[red](5,6) circle (4pt);
\fill[red](5,2) circle (4pt);
\fill[red](-3,10) circle (4pt);
\fill[red](1,12) circle (4pt);
\fill[red](5,10) circle (4pt);
\fill[red](5,-2) circle (4pt);
\fill[red](1,-4) circle (4pt);
\fill[red](-3,-2) circle (4pt);
\fill[green](-1,7) circle (4pt);
\fill[green](3,5) circle (4pt);
\fill[green](3,1) circle (4pt);
\fill[green](7,3) circle (4pt);
\fill[green](-5,9) circle (4pt);
\fill[green](-1,11) circle (4pt);
\fill[green](3,9) circle (4pt);
\fill[green](7,7) circle (4pt);
\fill[green](-5,5) circle (4pt);
\fill[green](7,-1) circle (4pt);
\fill[green](3,-3) circle (4pt);
\fill[green](-1,-1) circle (4pt);
\fill[green](-5,1) circle (4pt);
\end{tikzpicture} 
\caption*{type 6: $[1,3, 2^*]^{3} [2^*]$}
\end{subfigure} 
\hspace{0.3\textwidth}
\begin{subfigure}[b]{0.1\textwidth}
\begin{tikzpicture}[scale=.3]
\draw(2,0)--(4,2);
\draw(3,5)--(4,6);
\draw(4,6)--(2,8);
\draw(0,8)--(-2,6);
\draw(-2,6)--(-4,6);
\draw(-4,6)--(-6,4);
\draw(-6,4)--(-4,2);
\draw(-2,6)--(0,4);
\draw(-0.5,3.5)--(0,4);
\draw(4,6)--(6,6);
\draw(6,6)--(8,4);
\draw(8,4)--(6,2);
\draw(6,2)--(4,2);
\draw(6,6)--(8,8);
\draw(8,8)--(6,10);
\draw(6,10)--(4,10);
\draw(4,10)--(2,8);
\draw(4,10)--(2,12);
\draw(2,12)--(0,12);
\draw(0,12)--(-2,10);
\draw(-2,10)--(-4,10);
\draw(-2,10)--(0,8);
\draw(-4,10)--(-6,8);
\draw(-6,8)--(-4,6);
\draw(6,2)--(8,0);
\draw(8,0)--(6,-2);
\draw(6,-2)--(4,-2);
\draw(4,-2)--(2,-4);
\draw(2,-4)--(0,-4);
\draw(0,-4)--(-2,-2);
\draw(-2,-2)--(0,0);
\draw(-2,-2)--(-4,-2);
\draw(-4,-2)--(-6,0);
\draw(-6,0)--(-4,2);
\draw(2,0)--(4,-2);
\draw(2,0)--(0,0);
\draw(-1,1)--(0,0);
\draw(0,8)--(2,8);
\draw(-4,2)to[bend left](1,0);
\draw(4,2)to[out=-170,in=-140](-1,5);
\draw(0,4)to[bend right](5,6);
\fill[blue](3,7) circle (4pt);
\fill[blue](-5,3) circle (4pt);
\fill[blue](-1,5) circle (4pt);
\fill[blue](7,5) circle (4pt);
\fill[blue](-5,7) circle (4pt);
\fill[blue](-1,9) circle (4pt);
\fill[blue](3,11) circle (4pt);
\fill[blue](7,9) circle (4pt);
\fill[blue](7,1) circle (4pt);
\fill[blue](3,-1) circle (4pt);
\fill[blue](-1,-3) circle (4pt);
\fill[blue](-5,-1) circle (4pt);
\fill[blue](-1,1) circle (4pt);
\fill[red](-3,6) circle (4pt);
\fill[red](1,8) circle (4pt);
\fill[red](1,0) circle (4pt);
\fill[red](5,6) circle (4pt);
\fill[red](5,2) circle (4pt);
\fill[red](-3,10) circle (4pt);
\fill[red](1,12) circle (4pt);
\fill[red](5,10) circle (4pt);
\fill[red](5,-2) circle (4pt);
\fill[red](1,-4) circle (4pt);
\fill[red](-3,-2) circle (4pt);
\fill[green](-0.5,3.5) circle (4pt);
\fill[green](-1,7) circle (4pt);
\fill[green](3,5) circle (4pt);
\fill[green](3,1) circle (4pt);
\fill[green](7,3) circle (4pt);
\fill[green](-5,9) circle (4pt);
\fill[green](-1,11) circle (4pt);
\fill[green](3,9) circle (4pt);
\fill[green](7,7) circle (4pt);
\fill[green](-5,5) circle (4pt);
\fill[green](7,-1) circle (4pt);
\fill[green](3,-3) circle (4pt);
\fill[green](-1,-1) circle (4pt);
\fill[green](-5,1) circle (4pt);
\end{tikzpicture} 
\caption*{type 5: $[3^{2}, 2^*][1^{2}, 2^*][1,3, 2^*] [2^*]$}
\end{subfigure} 

\end{figure}

\begin{figure}[h!] 
\flushleft
\begin{subfigure}[b]{0.1\textwidth}
\begin{tikzpicture}[scale=.3]
\draw(2,0) -- (0,0);
\draw(2,0)--(4,2);
\draw(4,2)--(2,4);
\draw(2,4)--(4,6);
\draw(4,6)--(2,8);
\draw(0,8)--(-2,6);
\draw(-2,6)--(-4,6);
\draw(-4,6)--(-6,4);
\draw(-6,4)--(-4,2);
\draw(-4,2)--(-2,2);
\draw(-2,2)--(0,0);
\draw(-2,6)--(0,4);
\draw(-2,2)--(0,4);
\draw(4,6)--(6,6);
\draw(6,6)--(8,4);
\draw(8,4)--(6,2);
\draw(6,2)--(4,2);
\draw(6,6)--(8,8);
\draw(8,8)--(6,10);
\draw(6,10)--(4,10);
\draw(4,10)--(2,8);
\draw(4,10)--(2,12);
\draw(2,12)--(0,12);
\draw(0,12)--(-2,10);
\draw(-2,10)--(-4,10);
\draw(-2,10)--(0,8);
\draw(-4,10)--(-6,8);
\draw(-6,8)--(-4,6);
\draw(2,8)--(0,8);
\draw(2,4)--(0,4);
\draw(1,6)--(1,8);
\draw(1,6)--(3,5);
\draw(1,6)--(-1,5);
\fill[blue](-1,1) circle (4pt);
\fill[blue](3,7) circle (4pt);
\fill[blue](3,3) circle (4pt);
\fill[blue](-5,3) circle (4pt);
\fill[blue](-1,5) circle (4pt);
\fill[blue](7,5) circle (4pt);
\fill[blue](-5,7) circle (4pt);
\fill[blue](-1,9) circle (4pt);
\fill[blue](3,11) circle (4pt);
\fill[blue](7,9) circle (4pt);
\fill[red](-3,2) circle (4pt);
\fill[red](-3,6) circle (4pt);
\fill[red](1,0) circle (4pt);
\fill[red](5,6) circle (4pt);
\fill[red](5,2) circle (4pt);
\fill[red](-3,10) circle (4pt);
\fill[red](1,12) circle (4pt);
\fill[red](1,8) circle (4pt);
\fill[red](1,4) circle (4pt);
\fill[red](5,10) circle (4pt);
\fill[green](-1,3) circle (4pt);
\fill[green](-1,7) circle (4pt);
\fill[green](3,5) circle (4pt);
\fill[green](3,1) circle (4pt);
\fill[green](7,3) circle (4pt);
\fill[green](-5,9) circle (4pt);
\fill[green](-1,11) circle (4pt);
\fill[green](3,9) circle (4pt);
\fill[green](7,7) circle (4pt);
\fill[green](-5,5) circle (4pt);
\end{tikzpicture} 
\caption*{type 11: $[3, 2^*]^{3}[1^{3} ,2^*]$}
\end{subfigure} 
\hspace{0.3\textwidth}
\begin{subfigure}[b]{0.1\textwidth}
\begin{tikzpicture}[scale=.3]
\draw(2,0)--(4,2);
\draw(3,3)--(2,4);
\draw(2,4)--(4,6);
\draw(-2,6)--(-4,6);
\draw(-4,6)--(-6,4);
\draw(-6,4)--(-4,2);
\draw(-4,2)--(-2,2);
\draw(-2,2)--(4,2);
\draw(-2,6)--(0,4);
\draw(-2,2)--(0,4);
\draw(4,6)--(6,6);
\draw(6,6)--(8,4);
\draw(8,4)--(6,2);
\draw(6,2)--(4,2);
\draw(6,2)--(8,0);
\draw(8,0)--(6,-2);
\draw(6,-2)--(4,-2);
\draw(4,-2)--(2,-4);
\draw(0,-4)--(-2,-2);
\draw(-2,-2)--(-4,-2);
\draw(-4,-2)--(-6,0);
\draw(-6,0)--(-6.5,1);
\draw(2,0)--(4,-2);
\draw(-6,-4)--(-4,-2);
\draw(6,-2)--(8,-4);
\draw(8,-4)--(6,-6);
\draw(6,-6)--(4,-6);
\draw(4,-6)--(2,-4);
\draw(4,-6)--(2,-8);
\draw(0,-8)--(-2,-6);
\draw(-2,-6)--(0,-4);
\draw(-2,-6)--(-4,-6);
\draw(-4,-6)--(-6,-4);
\draw(1,-5)--(2,-4);
\draw(-2,-2)--(-4,2);
\draw(-6,-8)--(-4,-6);
\draw(0,4)--(2,4);
\draw(2,-8)--(0,-8);
\draw(-8,4)--(-6,4);
\draw(-8,4)--(-10,2);
\draw(-10,2)--(-8,0);
\draw(-8,0)--(-6,0);
\draw(-8,0)--(-10,-2);
\draw(-10,-2)--(-8,-4);
\draw(-8,-4)--(-6,-4);
\draw(-8,-4)--(-10,-6);
\draw(-10,-6)--(-8,-8);
\draw(-8,-8)--(-6,-8);
\draw(0,-4)--(2,0);
\fill[blue](3,3) circle (4pt);
\fill[blue](-5,3) circle (4pt);
\fill[blue](-1,5) circle (4pt);
\fill[blue](7,5) circle (4pt);
\fill[blue](-5,-5) circle (4pt);
\fill[blue](7,-3) circle (4pt);
\fill[blue](3,-5) circle (4pt);
\fill[blue](-1,-7) circle (4pt);
\fill[blue](7,1) circle (4pt);
\fill[blue](3,-1) circle (4pt);
\fill[blue](-9,1) circle (4pt);
\fill[blue](-9,-3) circle (4pt);
\fill[blue](-9,-7) circle (4pt);
\fill[blue](-1,-3) circle (4pt);
\fill[blue](1,2) circle (4pt);
\fill[blue](-5,-1) circle (4pt);
\fill[red](-3,2) circle (4pt);
\fill[red](-3,6) circle (4pt);
\fill[red](5,6) circle (4pt);
\fill[red](5,2) circle (4pt);
\fill[red](1,4) circle (4pt);
\fill[red](1,-5) circle (4pt);
\fill[red](1,-2) circle (4pt);
\fill[red](5,-2) circle (4pt);
\fill[red](-3,-2) circle (4pt);
\fill[red](5,-6) circle (4pt);
\fill[red](-3,-6) circle (4pt);
\fill[red](-7,4) circle (4pt);
\fill[red](-7,0) circle (4pt);
\fill[red](-7,-8) circle (4pt);
\fill[red](1,-8) circle (4pt);
\fill[red](-7,-4) circle (4pt);
\fill[green](-1,3) circle (4pt);
\fill[green](3,5) circle (4pt);
\fill[green](3,1) circle (4pt);
\fill[green](7,3) circle (4pt);
\fill[green](-5,5) circle (4pt);
\fill[green](7,-1) circle (4pt);
\fill[green](3,-3) circle (4pt);
\fill[green](-3,0) circle (4pt);
\fill[green](-6.5,1) circle (4pt);
\fill[green](7,-5) circle (4pt);
\fill[green](3,-7) circle (4pt);
\fill[green](-1,-5) circle (4pt);
\fill[green](-5,-3) circle (4pt);
\fill[green](-5,-7) circle (4pt);
\fill[green](-9,3) circle (4pt);
\fill[green](-9,-1) circle (4pt);
\fill[green](-9,-5) circle (4pt);
\end{tikzpicture} 
\caption*{type 9: $[1,2^*]^{3}[3^{3},2^*]$}
\end{subfigure} 

\end{figure}

\begin{figure}[h!] 
\flushleft
\begin{subfigure}[b]{0.1\textwidth}
\begin{tikzpicture}[scale=.3]
\draw(4,2)--(2,4);
\draw(-4,2)--(-2,2);
\draw(-2,2)--(0,0);
\draw(-2,2)--(0,4);
\draw(6,2)--(4,2);
\draw(6,2)--(8,0);
\draw(8,0)--(6,-2);
\draw(6,-2)--(4,-2);
\draw(4,-2)--(2,-4);
\draw(0,-4)--(-2,-2);
\draw(-2,-2)--(0,0);
\draw(-2,-2)--(-4,-2);
\draw(-4,-2)--(-6,0);
\draw(-6,0)--(-4,2);
\draw(-6,-4)--(-4,-2);
\draw(6,-2)--(8,-4);
\draw(8,-4)--(6,-6);
\draw(6,-6)--(4,-6);
\draw(4,-6)--(2,-4);
\draw(4,-6)--(2,-8);
\draw(0,-8)--(-2,-6);
\draw(-2,-6)--(0,-4);
\draw(-2,-6)--(-4,-6);
\draw(-4,-6)--(-6,-4);
\draw(1,-8)--(2,-4);
\draw(0,0)--(0,-4);
\draw(8,-8)--(6,-6);
\draw(8,-8)--(6,-10);
\draw(6,-10)--(4,-10);
\draw(4,-10)--(2,-8);
\draw(-2,-10)--(0,-8);
\draw(-2,-10)--(-4,-10);
\draw(-4,-10)--(-6,-8);
\draw(-6,-8)--(-4,-6);
\draw(2,-8)--(0,-8);
\draw(2,-12)--(4,-10);
\draw(2,-12)--(0,-12);
\draw(0,-12)--(-2,-10);
\draw(2,4)--(0,4);
\draw(4,-2)--(7,1);
\draw(-1,3)--(4,2);
\fill[blue](-1,1) circle (4pt);
\fill[blue](3,3) circle (4pt);
\fill[blue](-5,-5) circle (4pt);
\fill[blue](7,-3) circle (4pt);
\fill[blue](3,-5) circle (4pt);
\fill[blue](-1,-7) circle (4pt);
\fill[blue](7,1) circle (4pt);
\fill[blue](-1,-3) circle (4pt);
\fill[blue](-5,-1) circle (4pt);
\fill[blue](7,-7) circle (4pt);
\fill[blue](3,-9) circle (4pt);
\fill[blue](-5,-9) circle (4pt);
\fill[blue](-1,-11) circle (4pt);
\fill[red](-3,2) circle (4pt);
\fill[red](0,-2) circle (4pt);
\fill[red](5,2) circle (4pt);
\fill[red](1,4) circle (4pt);
\fill[red](5,-2) circle (4pt);
\fill[red](-3,-2) circle (4pt);
\fill[red](5,-6) circle (4pt);
\fill[red](-3,-6) circle (4pt);
\fill[red](5,-10) circle (4pt);
\fill[red](-3,-10) circle (4pt);
\fill[red](1,-8) circle (4pt);
\fill[red](1,-12) circle (4pt);
\fill[green](-1,3) circle (4pt);
\fill[green](7,-1) circle (4pt);
\fill[green](3,-3) circle (4pt);
\fill[green](-1,-1) circle (4pt);
\fill[green](-5,1) circle (4pt);
\fill[green](7,-5) circle (4pt);
\fill[green](3,-7) circle (4pt);
\fill[green](-1,-5) circle (4pt);
\fill[green](-5,-3) circle (4pt);
\fill[green](7,-9) circle (4pt);
\fill[green](-1,-9) circle (4pt);
\fill[green](-5,-7) circle (4pt);
\fill[green](3,-11) circle (4pt);
\end{tikzpicture} 
\caption*{type 12: $[3 ,2^*]^{3}[1^{4},3,2^*]$}
\end{subfigure} 
\hspace{0.3\textwidth}
\begin{subfigure}[b]{0.1\textwidth}
\begin{tikzpicture}[scale=.3]
\draw(4,2)--(2,4);
\draw(2,4)--(4,6);
\draw(4,6)--(2,8);
\draw(0,8)--(-2,6);
\draw(-2,6)--(-4,6);
\draw(-4,6)--(-6,4);
\draw(-6,4)--(-4,2);
\draw(-4,2)--(-2,2);
\draw(-2,2)--(0,0);
\draw(-2,6)--(0,4);
\draw(-2,2)--(0,4);
\draw(4,6)--(6,6);
\draw(6,6)--(8,4);
\draw(8,4)--(6,2);
\draw(6,2)--(4,2);
\draw(6,2)--(8,0);
\draw(8,0)--(6,-2);
\draw(6,-2)--(4,-2);
\draw(4,-2)--(2,-4);
\draw(-2,-2)--(0,0);
\draw(-2,-2)--(-4,-2);
\draw(-4,-2)--(-6,0);
\draw(-6,0)--(-4,2);
\draw(-6,-4)--(-4,-2);
\draw(6,-2)--(8,-4);
\draw(8,-4)--(6,-6);
\draw(6,-6)--(4,-6);
\draw(4,-6)--(2,-4);
\draw(4,-8)--(2,-8);
\draw(0,-8)--(-2,-6);
\draw(-2,-2)--(-4,-6);
\draw(-2,-6)--(-4,-6);
\draw(-5,-5)--(-6,-4);
\draw(8,-8)--(6,-6);
\draw(8,-8)--(6,-10);
\draw(6,-10)--(4,-10);
\draw(4,-10)--(2,-8);
\draw(4,2)--(2,0);
\draw(4,-2)--(2,0);
\draw(-2,-10)--(0,-8);
\draw(-2,-10)--(-4,-10);
\draw(-4,-10)--(-6,-8);
\draw(-6,-8)--(-4,-6);
\draw(2,-8)--(0,-8);
\draw(-8,4)--(-6,4);
\draw(-8,4)--(-10,2);
\draw(-10,2)--(-8,0);
\draw(-8,0)--(-6,0);
\draw(-8,0)--(-10,-2);
\draw(-10,-2)--(-8,-4);
\draw(-8,-4)--(-6,-4);
\draw(-8,-4)--(-10,-6);
\draw(-10,-6)--(-8,-8);
\draw(-8,-8)--(-6,-8);
\draw(2,-12)--(4,-10);
\draw(2,-12)--(0,-12);
\draw(0,-12)--(-2,-10);
\draw(-6,-12)--(-4,-10);
\draw(-6,-12)--(-8,-12);
\draw(-8,-12)--(-10,-10);
\draw(-10,-10)--(-8,-8);
\draw(2,8)--(0,8);
\draw(1,5)--(0,4);
\draw(2,4)--(0,0);
\draw(2,0)--(2,-4);
\draw(-2,-6)--(4,-6);
\fill[blue](-1,1) circle (4pt);
\fill[blue](3,7) circle (4pt);
\fill[blue](3,3) circle (4pt);
\fill[blue](-5,3) circle (4pt);
\fill[blue](-1,5) circle (4pt);
\fill[blue](7,5) circle (4pt);
\fill[blue](-5,-5) circle (4pt);
\fill[blue](7,-3) circle (4pt);
\fill[blue](3,-5) circle (4pt);
\fill[blue](7,1) circle (4pt);
\fill[blue](-9,1) circle (4pt);
\fill[blue](-9,-3) circle (4pt);
\fill[blue](-9,-7) circle (4pt);
\fill[blue](-5,-1) circle (4pt);
\fill[blue](7,-7) circle (4pt);
\fill[blue](3,-9) circle (4pt);
\fill[blue](-5,-9) circle (4pt);
\fill[blue](-1,-11) circle (4pt);
\fill[blue](-9,-11) circle (4pt);
\fill[blue](3,-1) circle (4pt);
\fill[blue](-3,-4) circle (4pt);
\fill[red](-3,2) circle (4pt);
\fill[red](-3,6) circle (4pt);
\fill[red](5,6) circle (4pt);
\fill[red](5,2) circle (4pt);
\fill[red](1,8) circle (4pt);
\fill[red](1,5) circle (4pt);
\fill[red](5,-2) circle (4pt);
\fill[red](-3,-2) circle (4pt);
\fill[red](2,-2) circle (4pt);
\fill[red](5,-6) circle (4pt);
\fill[red](5,-10) circle (4pt);
\fill[red](-3,-10) circle (4pt);
\fill[red](-7,4) circle (4pt);
\fill[red](-7,0) circle (4pt);
\fill[red](-7,-8) circle (4pt);
\fill[red](1,-8) circle (4pt);
\fill[red](-7,-4) circle (4pt);
\fill[red](-7,-12) circle (4pt);
\fill[red](1,-12) circle (4pt);
\fill[red](1,2) circle (4pt);
\fill[red](-3,-6) circle (4pt);
\fill[green](-1,3) circle (4pt);
\fill[green](3,1) circle (4pt);
\fill[green](-1,7) circle (4pt);
\fill[green](3,5) circle (4pt);
\fill[green](7,3) circle (4pt);
\fill[green](-5,5) circle (4pt);
\fill[green](7,-1) circle (4pt);
\fill[green](3,-3) circle (4pt);
\fill[green](1,-6) circle (4pt);
\fill[green](-1,-1) circle (4pt);
\fill[green](-5,1) circle (4pt);
\fill[green](7,-5) circle (4pt);
\fill[green](4,-8) circle (4pt);
\fill[green](-5,-3) circle (4pt);
\fill[green](7,-9) circle (4pt);
\fill[green](-1,-9) circle (4pt);
\fill[green](-5,-7) circle (4pt);
\fill[green](-9,3) circle (4pt);
\fill[green](-9,-1) circle (4pt);
\fill[green](-9,-5) circle (4pt);
\fill[green](-9,-9) circle (4pt);
\fill[green](3,-11) circle (4pt);
\fill[green](-5,-11) circle (4pt);
\end{tikzpicture} 
\caption*{type 10: $[1, 2^*]^{3}[1,3^{4},2^*]$}
\end{subfigure}

\end{figure} 


\section{Realizability of almost-regular ramification data on a sphere}
\label{sect:sphere}
In this section we prove Theorem \ref{thm:localch0}, but first we note several differences from the genus $1$ case. 
Firstly, families of almost-regular ramification data of genus zero are in general not realizable in all degrees, see Section \ref{sec:intro}. 
In particular, it should be noted that the argument of Proposition \ref{prop:localch1}, ensuring existence in all sufficiently large degrees, does not apply here. 
Secondly, the notion of regular ramification types is not of use here, since, as a consequence of the Riemann-Hurwitz formula, there are no infinite families of regular ramification data in genus zero. 
We therefore provide a pure existence result (Theorem \ref{thm:localch0}) for the case of genus $0$. See however the appendix, where we develop terminology generalizing the notions of regular dessins and of local changes, and give some evidence for strengthening of Theorem \ref{thm:localch0}.

 
%


\begin{proof}[Proof of Theorem \ref{thm:localch0}]
A straightforward calculation shows that all of the genus 0 almost-regular ramification data satisfying the Riemann-Hurwitz formula with $\varepsilon \le 10$ from the types $[3^*]^3$ or $[2^*]^4$, or with $\varepsilon \le 6$ from $[2^*][3^*][6^*]$ or $ [2^*][4^*][4^*]$  are exactly the ones in the following list: 

\begin{paracol}{2} 
\begin{tabular}{ccc}
(1) & [NR] &$[1, 3^*] [1, 3^*] [1, 3^*]$\\
(2) & [R1]  
&$[ 3^*] [ 3^*] [1,1,1, 3^*]$ \\
(3) & [NR] & $[1, 3^*] [1, 3^*] [2,2, 3^*]$\\
(4) & [NR] & $[1,1, 3^*] [2, 3^*] [2, 3^*]$\\
(5) & [NR]  &$[ 3^*] [ 1,2,3^*] [1,2, 3^*]$\\ 
(6) & [R4]  
&$[ 3^*] [ 3^*] [1,1,2,2, 3^*]$ \\
(7) & [NR]  &$[1, 3^*] [1, 3^*] [1,1,5, 3^*]$\\
(8) & [NR]  &$[1, 3^*] [1, 3^*] [1,2,4, 3^*]$ \\
(9) & [NR]  &$[1, 3^*] [1,1,2, 3^*] [4, 3^*] $\\
(10) & [NR] & $[1, 3^*] [2,2, 3^*] [2,2, 3^*]$ \\
(11) & [NR]  &$[1,2,2, 3^*] [2, 3^*] [2, 3^*]$\\  
(12) & [R1]  
&$[1,1, 3^*] [1,4, 3^*] [2, 3^*]$ \\
(13) & [NR]  &$[1,1, 3^*] [1,1, 3^*] [5, 3^*]$ \\
(14) & [R13] 
&$[ 3^*] [ 3^*] [1^4, 5, 3^*]$ \\
(15) & [R12]  
& $[ 3^*] [ 3^*] [1^3, 2,4, 3^*]$ \\
(16) & [R10]  
& $[ 3^*] [ 3^*] [1,2^4, 3^*]$ \\
(17) & [NR] & $[ 3^*] [1,1,4, 3^*] [1,2, 3^*]$\\ 
(18) & [NR] & $[ 3^*] [ 1,2,3^*] [2^3, 3^*]$ \\
(19) & [NR]  &$[ 3^*] [ 1^3, 3^*] [1,5, 3^*]$ \\
(20) & [NR] & $[ 3^*] [ 1^3,3^*] [2,4, 3^*]$ \\
& & \\
(21) &[NR] &$[1,2^*] [1,3^*] [1,6^*] $\\
(22) &[R2] 
&$[2^*] [1,1,1,3^*] [6^*] $\\
(23) &[R47] 
&$[1^4, 2^*] [3^*] [6^*] $\\
(24) &[R1] 
&$[2^*] [1,1,3^*] [2,6^*] $\\
(25) &[NR]& $[2^*] [2,3^*] [1,1,6^*] $\\
(26) & [NR]& $[1,2^*] [3^*] [1,2,6^*] $\\
(27) &[NR] &$[1,1,2^*] [1,2,3^*] [6^*] $\\
(28) & [NR]& $[2^*] [1,3^*] [1,3,6^*]$ \\
(29) & [NE] &$[2^*] [1,3^*] [2,2,6^*] $\\
(30) &[R25] 
&$[2^*] [3^*] [1,1,4,6^*] $\\
(31) &[NR] &$[2^*] [3^*] [1,2,3,6^*] $
\end{tabular}
\begin{tabular}{ccc}
(32) &[R2]  
&$[2^*] [3^*] [2,2,2,6^*] $\\
(33) &[R5]  
&$[2^*] [1,1,2,2,3^*] [6^*] $\\
(34) &[NR] &$[1,2^*] [2,2,3^*] [1,6^*] $\\
(35) &[NR] &$[1,1,2^*] [2,3^*] [2,6^*] $\\
(36) &[R25]  
&$[3,2^*] [3^*] [1^3,6^*] $\\
(37) & [R45] 
&$[1^3, 2^*] [3^*] [3,6^*] $\\
& & \\
(38) & [NR] &$[1,2^*][1,4^*][1,4^*] $\\
(39) & [R47]  
&$[1^4,2^*][4^*][4^*] $\\
(40) & [NR]&$[2^*][4^*][1,1,2,4^*] $\\ 
(41) & [R38]  
&$[2^*][2,4^*][1,1,4^*] $ \\
(42) & [NR] &$[1,1,2^*][4^*][1,3,4^*] $\\ 
(43) & [R47]  
&$[1,1,2^*][4^*][2,2,4^*] $\\  
(44) & [R45] 
&$[1,1,2^*][2,4^*][2,4^*] $ \\
& & \\
(45) &[NR] &$[1,2^*][1,2^*][1,2^*][1,2^*]$\\
(46) &[NR] &$[2^*][2^*][2^*][1^4,2^*]$\\
(47) &[R45] 
&$[2^*][2^*][1,1,2^*][1,1,2^*]$\\
(48) &[NR] &$[1,2^*][1,2^*][1,2^*][1,1,3,2^*]$\\
(49) &[R46] 
&$[1,2^*][1,2^*][1^3,2^*][3,2^*]$\\
(50) &[R48] 
&$[2^*][2^*][1,1,2^*][1^3,3,2^*]$\\
(51) &[R49] 
&$[2^*][2^*][1^4,2^*][1,3,2^*]$\\
(52) &[R45] 
&$[2^*][1,1,2^*][1,1,2^*][1,3,2^*]$\\
(53) &[NR] &$[1^5,3,2^*][2^*][2^*] [2^*] $ \\
(54) &[NR] &$[1,2^*][1,2^*][1^3,2^*][1,4,2^*]$\\
(55) &[R58] 
&$[2^*][2^*][1^6,2^*][4,2^*]$\\
(56) &[R59] 
&$[2^*][2^*][1^4,2^*][1,1,4,2^*]$\\
(57) &[NR] &$[2^*][1^2 ,2^*][1^2 ,2^*][1^2,4,2^*]$\\
(58) &[NR] &$[2^*][1,1,2^*][1^4,2^*][4,2^*]$\\
(59) &[NR] &$[1^2 ,2^*][1^2 ,2^*][1^2 ,2^*][4,2^*]$\\
(60) &[NR] &$[1^4,4,2^*][2^*][2^*][1^2,2^*]$\\
(61) &[NR] &$[1^6,4,2^*][2^*][2^*] [2^*]$\\
(62) &[NR] &$[1^3,4,2^*][1,2^*][1,2^*][1,2^*]$ 
\end{tabular} 
\end{paracol} 

Once again we  label R for ``reduction" and add the number of the type that we reduce to, label NR for ``no reduction given", and NE for ``non-existent".

For all types marked ``NR" see the drawings in the pdf companion. 
It is easy to see that all of them can be drawn iteratively in infinitely many degrees by adding rings of hexagons around the middle. \\
The reduction arguments indicated above are carried out as in Section \ref{sect:class_1}. Here a reduction to a family that has already been realized in infinitely many degrees yields the same result for the family marked ``R". With the exception of types 12, 36, 49, 52 where the reduction is via Lemma \ref{lemma:addlines}, the reductions are all done  by composing with a genus $0$ covering. We carry out two of the reductions deserving special attention explicitly. 
\\
Firstly, the type 30 can be obtained in the following way: Let $h:\mathbb{P}^1\to \mathbb{P}^1$ be defined by $x\mapsto x^2(x-1)$. This covering has ramification type $[2,1][2,1][3]$. Denote by $a_{i,j}$ the unique preimage of the $i$-th branch point of $h$ with multiplicity $j$, for $1\le i,j\le 2$. Let $g$ be a covering of ramification type $[2^*][2,3^*][1,1,6^*]$, whose branch points (in this order) are $a_{1,1}, a_{2,2}$ and $a_{2,1}$. Then $h\circ g$ has ramification type $[2^*][1,1,4,6^*][3^*]$.
\\
Secondly, Lemma \ref{lemma:addlines} yields type 52 from type 45 by adding an edge between a vertex of degree $2$ and an adjacent one of degree $1$. To make sure that there are two such adjacent vertices, it suffices to observe that if all vertices of degree $1$ were connected to each other, they would form their own connected component, contradicting that the whole dessin is connected.

Finally, we show that the type $[2^*] [1 , 3^*] [ 2,2, 6^*]$ is nonrealizable. This argument was also previously used by Zieve to show the nonrealizability of this type. 
We first claim that this ramification type has to correspond to a decomposable covering. Indeed, call $P_1,P_2, P_3$ the points over which we have ramification $[2^*], [1, 3^*], [2,2,6^*],$ respectively. By Lemma 9.1.1 in \cite{NZ}, the function has to factor through a degree 2 function with ramification $[2] [2]$ over $P_1,P_3$, proving the claim. However, since there is no way to split $[1,3^*]$ among the two preimages of $P_2$, this ramification data does not occur for decomposable coverings either.
\end{proof} 

We conclude this section by noting that, as the error is increased, one can construct infinitely many almost-regular families of spherical ramification types which are not realizable, using the same Lemma 9.1.1 in \cite{NZ} as above. Compare once again with the situation in genus $1$, where conjecturally there are only four exceptional families. Furthermore, we expect that as we allow more than six changes to $[2,3,6]$ and $[2,4,4]$ types, there will be more ramification types which are realizable only as quasi-local changes. 
%
%
%



\section{Proofs of Nonexistence of Almost-Regular Ramification Types on a Torus}
\label{nonex:genus1}

In this section we prove following genus-$1$ ramification types are not realizable. 
\begin{paracol}{4} 
\begin{tabular}{cccc} 
 $ [3^*] [3^*] [2, 4, 3^*]$ &
 $ [2^*] [2^*] [2^*] [1,3, 2^*]$ &
 $[ 2^*] [4^*][ 3,5, 4^*]$ &
 $ [ 2^*] [ 3^*] [5,7,6^*]$
\end{tabular} 
\end{paracol}

Note the last two on this list have $\varepsilon \geq 7$, but they are shown not to exist in \cite{IKRSS}, so we include them here for completeness. 

Corjava and Zannier \cite{CZ} have provided a non-existence argument for the first type.  The proof uses a counting argument for divisor classes in the Picard group of an elliptic curve. 



The next argument shows nonexistence of $ [ 2^*] ^3 [1,3, 2^*]$, proven independently by Do-Zieve \cite{Z}. 

\begin{theorem}
The ramification type $[2^*]^3[1,3,2^*]$ is not realizable in any degree.
\end{theorem}
\begin{proof}
Since every dessin is induced by a covering, it suffices to prove the following statement:\\
Let $n\ge 4$ be an even integer. There are no permutations $\sigma_1,...,\sigma_4\in S_n$ such that all of the following hold:
\begin{itemize}
\item[a)] $\sigma_1$ is of cycle type $[1,2^*,3]$,
\item[b)] $\sigma_2,...,\sigma_4$ are all of cycle type $[2^*]$,
\item[c)] $\sigma_1\cdots \sigma_4=1$.
\end{itemize}

Assume that such $\sigma_1,...,\sigma_4$ exist. By c), we have $\sigma_1\sigma_2=(\sigma_3\sigma_4)^{-1}$, so $\sigma_1\sigma_2$ and $\sigma_3\sigma_4$ have the same cycle type. We will investigate this cycle type, first by looking at $\sigma_3$ and $\sigma_4$, and then by looking at $\sigma_1$ and $\sigma_2$.

Firstly, note that when investigating the cycles of a product $xy$, we can treat each orbit of the group $\langle x,y\rangle$ separately, because every orbit of the latter group is of course a union of orbits of (the cyclic group generated by) $xy \in \langle x,y\rangle$, i.e.\ a union of cycles of $xy$.

Now $\sigma_3\sigma_4$ is a product of two fixed-point free involutions. Restrict the action of $\langle \sigma_3,\sigma_4\rangle$ to one of its orbits. Denote the images of $x\in \langle\sigma_3,\sigma_4\rangle$ on this orbit by $\overline{x}$. Of course, $\overline{\sigma_3}$ and $\overline{\sigma_4}$ are still fixed-point free involutions. Note also that $\langle \overline{\sigma_3},\overline{\sigma_4}\rangle$ is transitive. But a transitive group generated by two involutions is a dihedral group $D_{2k}$ (of order $2k$ for some $k\in \NN$), and the fact that the involutions act without fixed points forces the product $\overline{\sigma_3\sigma_4}$ to have exactly two cycles of length $k$.

{\it This last claim is of course easy to see in a purely combinatorial way, but can also be shown elegantly using dessins: As a special case of Riemann-Hurwitz formula, the element $\overline{\sigma_3\sigma_4}$ has to have exactly two cycles, and the genus of $(\overline{\sigma_3},\overline{\sigma_4},(\overline{\sigma_3\sigma_4})^{-1})$ is then zero (no other choice for $\overline{\sigma_3\sigma_4}$ would make this genus a non-negative integer). The corresponding dessin is then a connected planar graph with all vertices of degree 2, i.e.\ a circle. Therefore the two faces have the same degree $k$.}

The above shows that every cycle length of $\sigma_1\sigma_2$ occurs an even number of times.

We now investigate the orbits of $U:=\langle \sigma_1,\sigma_2\rangle$. Without loss of generality, assume that the $3$-cycle of $\sigma_1$ is $(1,2,3)$, and the fixed point is $4$.\\
First, note that all orbits of $U$ are of even length (since these orbits are unions of cycles of $\sigma_2$, and the latter has only cycles of length 2), and are unions of cycles of $\sigma_1$. In particular, $1,2,3$ and $4$ have to be contained in the same orbit of $U$. Denote this orbit by $\mathcal{O}$. The image of $\sigma_1$ on $\mathcal{O}$ is then still of cycle type $[1,2^*,3]$, and the image of $\sigma_2$ is still of type $[2^*]$. As above, denote images by $\overline{x}$, and note that $\overline{U}$ acts transitively.
The Riemann-Hurwitz formula then forces $\overline{\sigma_1\sigma_2}$ to have exactly two cycles (and the tuple $(\overline{\sigma_1},\overline{\sigma_2},(\overline{\sigma_1\sigma_2})^{-1})$ is of genus zero by the Riemann-Hurwitz formula). But now there are only two cases, each of which will lead to a contradiction:
\begin{itemize}
\item[i)] $\overline{\sigma_1\sigma_2}$ has two cycles of the same length $k$.\\
We then obtain a planar dessin of type $([1,2^*,3], [2^*], [k,k])$. However such a dessin does not exist (see e.g.\ \cite[Lemma 4.7]{Pakovitch}).
\item[ii)] $\overline{\sigma_1\sigma_2}$ has two cycles of different lengths $k_1\ne k_2$.\\
But note that on all other orbits of $U$ (except for $\mathcal{O}$), $\sigma_1$ and $\sigma_2$ both act as fixed point free involutions; so by the same argument as above, every cycle length of $\sigma_1\sigma_2$ in the action on any of these orbits occurs an even number of times. Adding the one missing orbit $\mathcal{O}$, we find that the cycle length $k_1$ (and of course also $k_2$) occurs an odd number of times in $\sigma_1\sigma_2$. But we showed above that all cycle lengths of $\sigma_1\sigma_2$ occur an even number of times. This contradiction ends the proof.
\end{itemize}
\end{proof}




For geometric insight, we summarize the proof of \cite{IKRSS} for nonexistence of $ [3^*] [3^*] [2, 4, 3^*]$.  The same idea may be used to proved nonexistence of $[ 2^*] [4^*][ 3,5, 4^*]$ and $ [ 2^*] [ 3^*] [5,7,6^*]$.    

Let $\mathbb{T}$ be a torus with a hexagonal tiling. We put an equilateral metric on $\mathbb{T}$ by declaring that every edge has length 1, and every hexagon is a Euclidean regular hexagon.    If there were a $ [3^*] [3^*] [2, 4, 3^*]$ tiling of the torus, it would be a regular tiling of hexagons with three hexagons meeting at each vertex, except at two special vertices, one with two hexagons meeting, and one with 4 hexagons.  This would induce a euclidean metric with exactly two cone singularities,
in which every point has a neighborhood either isometric to an open subset of the euclidean plane, or to the apex of a euclidean cone. 

In \cite{IKRSS}, a contradiction is obtained by studying the holonomy groups of such metrics, and proving the holonomy theorem which states that in a Euclidean cone metric with two cone points, of curvature $\pm \frac{2 \pi}{n}$, the holonomy group $H$ contains the cyclic group $C_n$ of order $n$ as a proper subgroup $C_n  \lneqq H$. 


The holonomy group of a surface, $M$, with Euclidean cone metric is generated by $\pi_1 (M)$ together with loops around each of the cone points.  
It is shown in \cite{IKRSS} that any hexangulation of $M$ with vertices that can be two-colored has holonomy group that is a subgroup of $C_3$.   We can also see this by looking at the developing map into the Euclidean plane, and noticing that the tiling has a 3-symmetry. 
However, the assumptions on the extra vertices of $ [3^*] [3^*] [2, 4, 3^*]$ imply $C_3$ is a proper subgroup of the holonomy group ($= C_3$), which is a contradiction.  The idea of the proof of the holonomy theorem is to find additional loops in $\pi_1 M$ by pushing paths into the digon formed by the geodesics between the two cone points.  

This argument works only for types which have changes only in one partition, as this condition is necessary to define a Euclidean cone metric. 




\begin{remark}  
Any tiling of the torus gives an infinite tiling of the plane, as the torus is a quotient of the plane.   To extend a tiling on the torus with local changes to a tiling of the plane with the same local changes, fill in the rest of the lattice outside the fundamental domain with a tiling of regular hexagons (or squares).  However, Dress \cite{Dress} has proved there are small changes to hexagon and square tilings, which tile the plane but not the torus.  For example, it \emph{is possible} to tile the plane with four edges meeting at every vertex, with a tiling of all squares except one pentagon and one triangle. As shown above, such a tiling, of ramification type $[2^*][4^*][3,5,4^*]$,  is not possible on the torus. 
\end{remark}

\section{Appendix: Quasi-local changes and stability in permutations}\label{sec:quasi-local}
Section \ref{sect:sphere} contains a pure existence result for almost-regular ramification types of genus $0$. It is natural to try to prove a ``local changes" property as well, as in Theorem \ref{thm:existence_strong}. To do so, we first have to generalize the notion of regular dessins.
\subsection{Regular Spherical Types as Quotients of a Regular Type on a Torus}
\label{sect:regular_types}
For any degree-$n$ covering $f: \mathbb{P}^1\to \mathbb{P}^1$ of ramification type $[s_1,...,s_r]$ (where the $s_i$ are partitions of $n$), let $g$ be the Galois closure of $f$, i.e.\ the (unique) Galois covering $g$ of minimal degree factoring as $g=f\circ h$ for some covering $h$. 
We call the ramification type $[s_1,...,s_r]$ a \textit{regular} spherical type, if $g$ is of one of the four regular ramification types of genus one. It is well-known that for a Galois covering, all ramification indices over one given branch point $p_i$ are the same, and in the above scenario are in fact equal to $\lcm(s_i):=\lcm(e_{i,1},...,e_{i,j})$ where $[e_{i,1},...,e_{i,j}]$ is the partition $s_i$. 

A regular spherical ramification type is a special case of an almost-regular ramification type. 
With this background, we can rephrase our above notion of regular ramification types:

\begin{lemma}
Assume that a covering $f:\mathbb{P}^1\to \mathbb{P}^1$ of ramification type $[s_1,...,s_r]$ exists. Then the following are equivalent:
\begin{itemize}
\item[i)] This type is a regular ramification type, with Galois closure of ramification type $[k_1^*,...,k_r^*]$ (one of $[2^*]^4$, $[3^*]^3$, $[2^*][4^*]^2$, or $[2^*][3^*][6^*]$).
\item[ii)] For all $i$, $\lcm(s_i)=k_i$.
\item[iii)] 
There is a Galois covering $g:\mathbb{T}\to \mathbb{P}^1$ from the torus to the sphere and a subgroup $U$ of the corresponding Galois group, such that $f: \mathbb{T}/U\to \mathbb{P}^1$ is the induced map on the quotient space.
\end{itemize}
\end{lemma}

Note that our definition of regular spherical ramification types is a natural analog of regular ramification types in genus $1$. Namely, by the Riemann-Hurwitz formula, the coverings from the torus to the sphere whose Galois closure is still of genus $1$ are exactly those of regular ramification type.


{\emph{Examples.}} There are several different regular spherical types, which may be drawn in non (graph) isomorphic ways. E.g.
$$ [1 ,2^*] ^4 \qquad  [ 1^4 ,2^*] [ 2^*] ^3 \qquad [ 1, 1 ,2^*] ^2 [ 2^*]^2$$
are all of the regular almost-regular types $[2,2,2,2]$. They can also be constructed more geometrically: Consider the $[2^*]$ tiling of the torus.  This tiling has a 2-symmetry. The symmetry matches that of an individual hexagon: the two points of each color are opposite each other.  
This symmetry corresponds to a deck transformation of order two on the universal cover which is a tiling of the plane. 
Depending on the number of hexagons in the original torus tiling, the quotient by this 2-symmetry is one of the three possible regular spherical tilings for $[2^*]^4$.
Similarly, we can see the regular spherical tilings for the other types also as quotients by a symmetry.  The $[3^*]^3$ tiling has a quotient by a 3-symmetry, to give a regular spherical dessin, see  figure \ref{reg_sph}. 
The $ [2^*][ 3^*][ 6^*]$ tiling has a quotient by a 6- symmetry, and the $[2^*][ 4^*][ 4^*]$ tiling has a quotient by a 4-symmetry to give the respective regular spherical dessins. 

With this view point, we want to see the ramification data as changes to these regular spherical types.  However, we must suitably generalize our notion of local changes. 

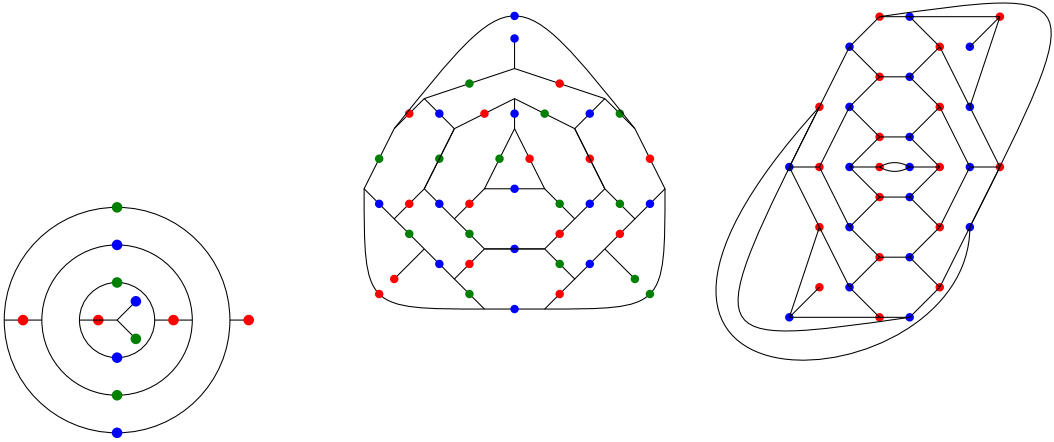
\begin{figure} 
\flushleft
\begin{subfigure} [t] {.2\textwidth} 
\begin{tikzpicture} [scale=.5]
\fill [red] (-1/2,0) circle (4pt) ; 
\draw (0,0) circle (1cm); 
\draw (-1,0) -- (0,0); 
\draw (0,0)-- (1/2, 1/2) ; 
\draw (0,0) -- (1/2, -1/2) ;
\fill [green] (1/2, -1/2) circle (4pt); 
\fill [blue] (1/2, 1/2) circle (4pt); 
\draw (0,0) circle (2cm); 
\draw (1,0)-- (2,0); 
\fill [red] (1.5,0) circle (4pt) ;
\fill [green] (0, 1) circle (4pt); 
\fill [blue] (0, -1) circle (4pt); 
\fill [green] (0, -2) circle (4pt); 
\fill [blue] (0, 2) circle (4pt); 
\draw (0,0)  circle (3cm); 
\draw ( -2,0)--(-3,0); 
\fill [red] (-2.5,0) circle (4pt) ;
\fill [green] (0, 3) circle (4pt); 
\fill [blue] (0, -3) circle (4pt); 
\draw (3,0)-- (3.5,0); 
\fill [red] (3.5,0) circle (4pt) ;
\end{tikzpicture} 
\label{$[1, 2^3] [1, 2^3] [1, 2^3] [1, 2^3]$}
\end{subfigure} 
\hspace{-1cm}
\begin{subfigure}[t] {.2\textwidth} 
\begin{tikzpicture}[scale=.4] 
\draw ( 0,0) -- (2,0) -- (1,2) --cycle; 
\fill[ blue] (1,0) circle (4pt); 
\fill [red ] ( 3/2, 1) circle (4pt); 
\fill [ green] (1/2,1) circle (4pt); 
\draw ( 1,2) -- ( 1,3) ; 
\fill [ blue] (1,5/2) circle (4pt); 
\draw ( 2,0) -- ( 3,-1) ; 
\fill [green] (5/2, -1/2) circle (4pt); 
\draw ( 0,0) -- (-1,-1) ; 
\fill [red] (-1/2, -1/2) circle (4pt) ; 
\draw ( 0, -2) -- (2, -2) ; 
\draw ( 0, -2 ) -- ( -1, -1) ; 
\draw ( 2, -2) -- ( 3, -1) ; 
\draw ( 3,-1) -- (4,0) ; 
\draw ( -1, -1) -- (-2,0); 
\draw ( 1,3) -- ( -1, 2) ; 
\draw (1,3) -- ( 3,2) ; 
\draw ( 3,2) -- ( 4,0) ; 
\draw ( -1, 2) -- ( -2,0) ; 
\fill [red] (5/2, -3/2) circle (4pt) ; 
\fill [ blue] (1,-2) circle (4pt);
\fill [green] (-1/2, -3/2) circle (4pt); 
\fill [ blue] (3.5,-1/2) circle (4pt);
\fill [red] (7/2, 1) circle (4pt) ; 
\fill [green] (2, 5/2) circle (4pt); 
\fill [red] (0, 5/2) circle (4pt) ; 
\fill [green] (-3/2, 1) circle (4pt); 
\fill [ blue] (-3/2,-1/2) circle (4pt);
\draw (0,-2) -- ( -1, -3) -- (0, -4) -- (2, -4) -- (3, -3) -- ( 2, -2) --cycle; 
\draw (3,2) -- ( 4,3) -- ( 5,2) -- (6,0) -- (5, -1) -- (4,0)-- cycle; 
\draw ( 3,-3) -- (4, -2) -- (5, -1) ; 
\draw (-2,0) -- (-3, -1) -- (-4, 0)-- ( -3, 2) -- ( -2, 3) -- ( -1, 2) -- (-2,0)-- cycle; 
\draw (-1, -3) -- (-2, -2) -- ( -3, -1) ; 
\draw ( -2, 3) -- (1, 4) -- ( 4, 3) ; 
\fill [green] (4.5, -1/2) circle (4pt); 
\fill [ blue] (3.5,2.5) circle (4pt);
\fill [ blue] (-3/2,2.5) circle (4pt);
\fill [red] (-2.5, -1/2) circle (4pt) ; 
\fill [red] (-.5, -2.5) circle (4pt) ; 
\fill [green] (2.5, -2.5) circle (4pt); 
\fill [red] (2.5, -3.5) circle (4pt) ; 
\fill [ blue] (1,-4) circle (4pt);
\fill [green] (-.5, -3.5) circle (4pt); 
\fill [ blue] (3.5,-2.5) circle (4pt);
\fill [red] (4.5, -1.5) circle (4pt) ; 
\fill [ blue] (5.5,-.5) circle (4pt);
\fill [red] (5.5, 1) circle (4pt) ;
\fill [green] (4.5, 2.5) circle (4pt); 
\fill [red] (2.5, 3.5) circle (4pt) ; 
\fill [green] (-.5, 3.5) circle (4pt); 
\fill [red] (-2.5, 2.5) circle (4pt) ; 
\fill [green] (-3.5, 1) circle (4pt); 
\fill [ blue] (-3.5,-.5) circle (4pt);
\fill [green] (-2.5, -1.5) circle (4pt); 
\fill [ blue] (-1.5,-2.5) circle (4pt);
\draw (1,4) -- (1,5) ; 
\fill [ blue] (1,5) circle (4pt);
\draw (-3,2) .. controls (1,7).. (5,2); 
\fill [ blue] (1,5.75) circle (4pt);
\draw (4,-2)-- (5,-3); 
\fill [green] (5, -3) circle (4pt); 
\draw (2,-4) ..controls (6, -4) .. ( 6,0); 
\fill [green] (5.5, -3.5) circle (4pt); 
\draw (-2,-2) -- (-3, -3) ; 
\fill [red] (-3, -3) circle (4pt) ; 
\draw (-4,0) .. controls (-4, -4) .. ( 0, -4) ; 
\fill [red] (-3.5, -3.5) circle (4pt) ; 
\node at (8,8) {}; 
\node at (-10, -8){}; 
\node at (30,0) {}; 
\end{tikzpicture} 
\end{subfigure}
\qquad \qquad
\begin{subfigure}[t] {.2\textwidth} 
\begin{tikzpicture}[scale=.4] 
\fill[red] (0,0) circle (4pt);
\fill[blue] (1,0) circle (4pt); 
\draw (0,0) to [bend right] (1,0) ; 
\draw (0,0) to [bend left] (1,0) ; 
\fill[red] (2,0) circle (4pt);
\fill[blue] (-1,0) circle (4pt); 
\draw (1,0) -- (2,0); 
\draw (0,0) -- (-1,0); 
\fill[blue] (1,1) circle (4pt); 
\fill[red] (0,1) circle (4pt);
\draw (1,1) --(0,1) ; 
\draw( 1,1) -- (2,0); 
\draw(0,1) -- (-1,0); 
\fill[red] (0,-1) circle (4pt);
\fill[blue] (1,-1) circle (4pt); 
\draw (0,-1) -- ( 1,-1) ; 
\draw (0,-1) -- (-1,0); 
\draw (1,-1) --(2,0); 
\fill[blue] (1,3) circle (4pt); 
\fill[red] (0,3) circle (4pt); 
\fill[red] (2,2) circle (4pt);
\fill[blue] (-1,2) circle (4pt); 
\draw (1,3) -- (0,3); 
\draw (1,3) -- (2,2) ; 
\draw(2,2) -- (1,1) ; 
\draw (0,3) -- ( -1,2) ; 
\draw (-1,2) --( 0,1) ; 
\fill[blue] (1,-3) circle (4pt); 
\fill[red] (0,-3) circle (4pt); 
\fill[red] (2,-2) circle (4pt);
\fill[blue] (-1,-2) circle (4pt); 
\draw (1,-3) -- (0,-3); 
\draw (1,-3) -- (2,-2) ; 
\draw(2,-2) -- (1,-1) ; 
\draw (0,-3) -- ( -1,-2) ; 
\draw (-1,-2) --( 0,-1) ; 
\fill[blue] (3,0) circle (4pt);
\fill[red] (-2,0) circle (4pt);  
\draw  (3,0) -- (2,-2) ; 
\draw ( 3,0) -- (2,2) ; 
\draw ( -2,0 ) -- (-1,2) ; 
\draw ( -2,0 ) -- ( -1,-2) ; 
\fill[blue] (1,5) circle (4pt); 
\fill[red] (2,4) circle (4pt); 
\fill[red] (0,5) circle (4pt);
\fill[blue] (-1,4) circle (4pt); 
\draw ( 1,3) -- ( 2,4) ; 
\draw ( 2, 4) -- (1,5) ; 
\draw ( 1,5) -- (0,5) ; 
\draw ( 0,5) -- (-1,4) ; 
\draw ( -1,4) -- (0,3) ; 
\fill[blue] (1,-5) circle (4pt); 
\fill[red] (2,-4) circle (4pt); 
\fill[red] (0,-5) circle (4pt);
\fill[blue] (-1,-4) circle (4pt); 
\draw ( 1,-3) -- ( 2,-4) ; 
\draw ( 2, -4) -- (1,-5) ; 
\draw ( 1,-5) -- (0,-5) ; 
\draw ( 0,-5) -- (-1,-4) ; 
\draw ( -1,-4) -- (0,-3) ;
\fill [red] (4,0) circle (4pt);  
\fill[ blue] (3,2) circle (4pt) ; 
\fill[ blue] (3,-2) circle (4pt) ; 
\draw (3,0) -- (4,0) ; 
\draw ( 4,0) -- ( 3,2) ; 
\draw ( 3,2) -- ( 2,4) ; 
\draw ( 4,0) -- ( 3,-2) ; 
\draw ( 3,-2) -- ( 2,-4) ;
\fill[ blue] (-3,0) circle  (4pt); 
\fill [red] (-2,2) circle (4pt) ; 
\fill [red] (-2,-2) circle (4pt) ; 
\draw ( -3,0 ) --(-2,0); 
\draw ( -3,0 ) -- ( -2,2) ; 
\draw ( -2,2) -- ( -1,4) ; 
\draw ( -3,0 ) -- ( -2,-2) ; 
\draw ( -2,-2) -- ( -1,-4) ; 
\fill[red] (4,5) circle (4pt) ; 
\draw ( 4,5) -- ( 1,5) ; 
\draw ( 4,5) -- ( 3,2) ; 
\fill [blue] (3,4) circle (4pt); 
\draw (3,4) -- (4,5) ; 
\fill [blue] (-3,-5) circle (4pt) ; 
\draw ( -3,-5) -- ( 0, -5) ; 
\draw ( -3,-5) -- ( -2, -2) ; 
\fill [ red] ( -2,-4) circle (4pt) ; 
\draw ( -2, -4) -- ( -3,-5) ; 
\draw ( 1,-5)  .. controls ( -6,-6) ..  (-2,2) ; 
\draw ( 0,5) .. controls(7,6)  .. (3,-2) ; 
\draw (3,-2) .. controls (3, -8) and (-12, -9) .. (-2, 2) ; 
\node at (8,8) {}; 
\node at (-10, -6){}; 
\end{tikzpicture} 
\end{subfigure} 
\caption{Regular $[2,2,2,2]$ and $[3,3,3]$ Spherical Dessins}
\label{reg_sph} 
\end{figure} 

\subsection{Quasi-local changes}

%
%

Below, we generalize the notion of local changes: 
\begin{definition}
\label{def:quasi_local}[Quasi-local changes]
Let $(C_i)_{i\in \NN}$ and $(D_i)_{i\in \NN}$ be two families of dessins on a surface $R$, all with the same set of branch points. Let $n_i$ be the degree of the covering associated to $C_i$, and $m_i$ be the degree of the covering associated to $D_i$. Assume that $n_i,m_i \to \infty$.
 We say that $(C_i)_i$ is realizable as \emph{quasi-local changes} to $(D_i)_i$ if $C_i$ can be transformed into $D_i$ with $k_i\in \NN$ changes where $\lim_{i\to \infty} k_i/n_i = 0$.
\end{definition}

\begin{question}
\label{conj:changes}
Assume that a family of almost-regular ramification types of genus zero is realizable. Is it then true that it is realizable as quasi-local changes to a family of dessins of regular spherical type?
\end{question}

We found Question \ref{conj:changes} to have a positive answer for all of the ramification types with the restrictions of Theorem \ref{thm:localch0}. In fact, it is not difficult to verify from our drawings that all the $[2,2,2,2]$-, $[3,3,3]$- and $[2,4,4]$-almost-regular types in that theorem are still realizable as local changes to a regular spherical dessin.
However, computational evidence suggests that the latter is no longer true for all $[2,3,6]$-almost-regular types, and in particular not for the type $[2^*][2,3^*][1,1,6^*]$.

In trying to answer Question \ref{conj:changes}, it may be useful to note that our reduction lemmas \ref{lemma:rattranslates} and \ref{lemma:addlines} still remain true (with the same proofs) after replacing the notion of local changes by quasi-local changes.
\subsection{Stability in permutations}
To give some more evidence why it is reasonable to expect Question \ref{conj:changes} (as well as Question \ref{conj:existence}) to have a positive answer, and in particular to justify the above notion of quasi-local changes, we relate them to a purely group-theoretical conjecture by Arzhantseva and Paunescu (\cite{AP}).

Recall the following, which is a version of the well-known Riemann existence theorem:
\begin{proposition}
Let $E:=(E_1,\dots,E_r)$ be a ramification data, where the $E_i$ are partitions of $n$. Then $E$ occurs as the ramification type of a covering of $\mathbb{P}^1$ if and only if there exist permutations $\sigma_1,...,\sigma_r\in S_n$ with the following properties:
\begin{itemize}
\item[i)] The cycle type of $\sigma_i$ is given by the partition $E_i$, for all $i=1,...,r$.
\item[ii)] $\sigma_1\cdots \sigma_r=1$.
\item[iii)] The group generated by $\sigma_1,\dots\sigma_r$ is a transitive subgroup of $S_n$.
\end{itemize}
\label{ret}
\end{proposition}
Due to the natural correspondence of dessins and tuples of permutations, the definition of local or quasi-local changes for families of dessins has a natural analog for
tuples of permutations. This notion is closely related to an existing notion of stability in permutations.
%
%
\begin{definition}[Stability]
\label{def:stability}
Let $F_m=\langle a_1,...,a_m\rangle$ be the free group of $m$ generators, and $R\subset F_m$ a finite set. For $\xi\in R$ and elements $x_1,...,x_m$ of some group $H$, denote by $\xi(x_1,...,x_m)$ the image under the unique homomorphism $F_m\to H$ with $a_i\mapsto x_i$.

Permutations $p_1,...,p_m\in S_n$ are called a solution of $R$ if $\xi(p_1,...,p_m)=1$ for all $\xi \in R$. They are called a $\delta$-solution of $R$ if $d_H(\xi(p_1,...,p_m),1)<\delta$ for all $\xi\in R$, where $d_H$ denotes the normalized Hamming distance in $S_n$, i.e. $d_H(p,q):=\frac{|\{x: 1\le x\le n, p(x)\ne q(x)\}|}{n}$.

The system $R$ is called stable (in permutations) if $\forall {\epsilon>0} \exists {\delta>0} \forall {n\in \NN}$ : For every $\delta$-solution $p_1,...,p_m \in S_n$ to $R$, there exists a solution $\tilde{p}_1,...\tilde{p}_m \in S_n$ to $R$ such that $d_H(p_i,\tilde{p}_i)<\epsilon$ (for all $i=1,...,m$).
The group $G=F_m/\langle R\rangle$ is called stable if $R$ is stable. 
\end{definition}

As noted in \cite{AP}, this definition is independent of the choice of the set $R$ of relators for a group $G$.  The notion of stability can be viewed as a group-theoretical analog of our notion of quasi-local changes for dessins. In the case of regular dessins of genus $1$, the associated permutation tuples as in Proposition \ref{ret} are of orders $(2,2,2,2)$, $(3,3,3)$, $(2,4,4)$ or $(2,3,6)$.
One may then ask whether, for a family of almost-regular dessins, the associated permutation tuples are close (in the sense of definition \ref{def:stability}) to some permutation tuples of the above orders.
Therefore, the groups we are interested in are the groups $\Delta(2,2,2,2)$, $\Delta(3,3,3)$, $\Delta(2,4,4)$ and $\Delta(2,3,6)$, where $\Delta(k_1,...,k_r):=\langle a_1,...,a_r\mid a_i^{k_i}=1 \text{ for } i=1,...,r\rangle$. The structure of these four groups is well known: They all have a normal subgroup isomorphic to $\ZZ^2$, and with cyclic quotient group (of order $2$, $3$, $4$ and $6$ in the respective cases).
In particular, all these groups are (polycyclic-by-finite and therefore) residually finite and (solvable and therefore) amenable.
It follows that all these groups (along with all their quotients) are ``weakly stable" in the sense of \cite[Def.\ 7.1]{AP}, see Theorem 1.1 of that paper.

It is conjectured in that paper that a group all of whose quotients are weakly stable is in fact stable (\cite[Conjecture 1.2]{AP}). Under this conjecture, the above groups would indeed all be stable. This would then imply that for any family of dessins of ``almost-regular" ramification type (say, close to a regular type $[k_1^*]...[k_r^*]$), it is possible to transform the corresponding permutation tuple into a tuple $(\sigma_1,...,\sigma_r)$ with $ord(\sigma_i)=k_i$ for all $i=1,...,r$, with ``quasi-local" changes in the sense of Definition \ref{def:stability}. (Note however, that this notion for permutation tuples is still slightly weaker than our notion of quasi-local changes for families of dessins; in particular, the group $\langle \sigma_1,...,\sigma_r\rangle$ need not be transitive; and if it is transitive, then the genus of the corresponding dessin need not be the same as the genus of the almost-regular type.)

\includepdf[pages={1-8}]{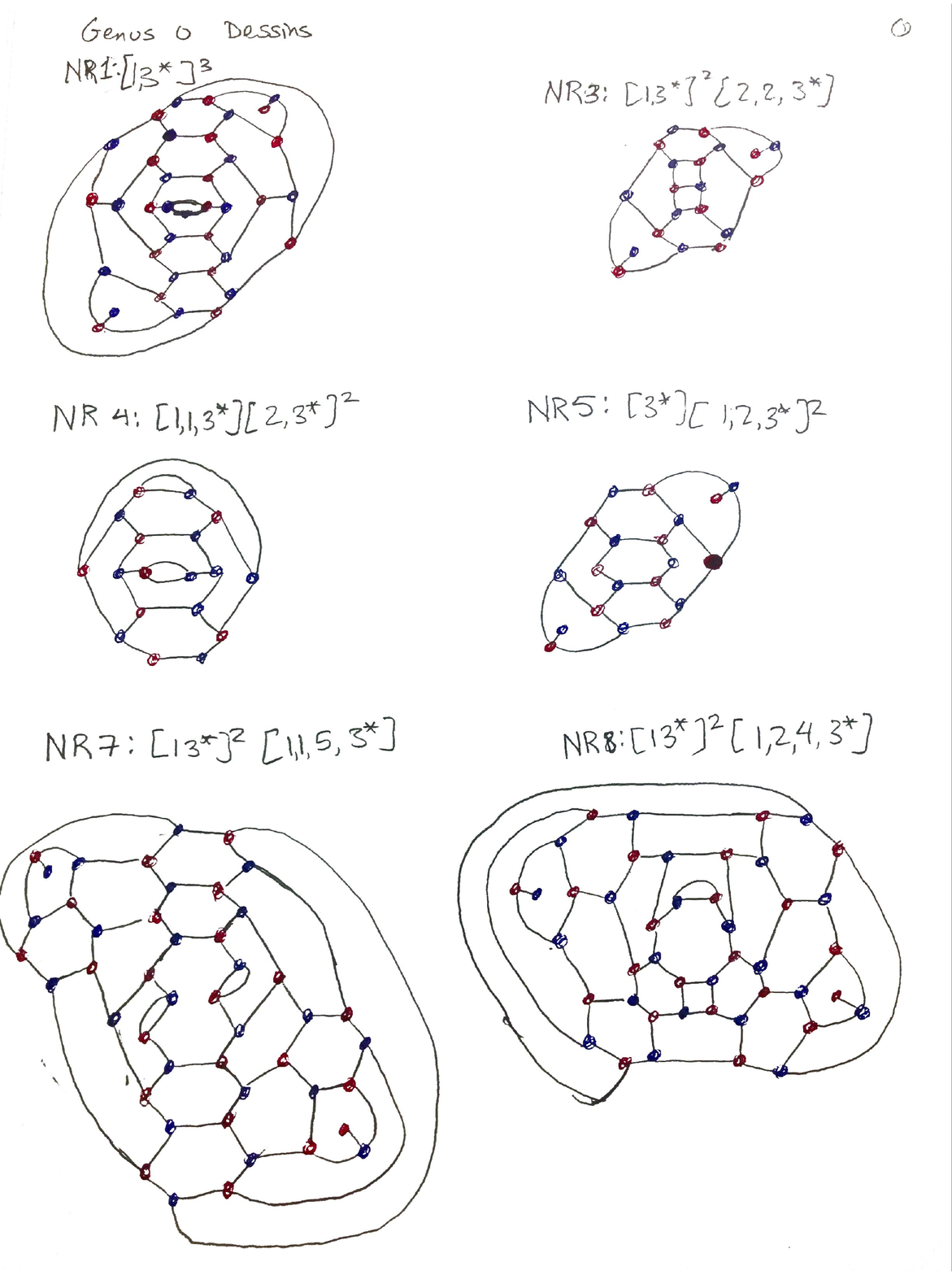}


\begin{thebibliography}{9}

\bibitem{AP} Arzhantseva, G., Paunescu, L.,
\textit{Almost commuting permutations are near commuting permutations}.
J.\ Functional Analysis 269, Issue 3 (2) (2015), 745--757.

\bibitem{Baranski} Baranski, K.,
\textit{On realizability of branched coverings of the sphere}.
Topology Appl.\ 116 (2001), no.\ 3, 279--291.

\bibitem{CZ}
Corvaja, P., Zannier, U.,
\textit{On the existence of covers of $P_1$ associated to certain permutations.} Preprint. 

\bibitem{Dress} Dress, A. \textit{On the classification of local disorder in globally regular spatial patterns}.  Temporal Order, Springer Ser. Synergetics, vol 29, 61--66. Spring, Berlin (1985). 

\bibitem{EKS} Edmonds, A.L., Kulkarni, R.S., Stong, R.E.,
\textit{Realizability of branched coverings of surfaces}.
Trans.\ Amer.\ Math.\ Soc.\ 282 (2)  (1984), 773--790.

\bibitem{GS}  Grunbaum, B., Shephard, G. \textit{Tilings and Patterns} W.H. Freeman and Company, New York, New York 1987. 

\bibitem{Hurw} Hurwitz, A.,
\textit{\"Uber Riemann'sche Fl\"achen mit gegebenen Verzweigungspunkten}.
Math.\ Ann.\ 39 (1891), 1--60.

\bibitem{IKRSS} Izmestiev, I., Kusner, R., Rote, G., Springborn, B., and Sullivan, J.M. \textit{ There is no triangulation of the torus with vertex degrees 5, 6, ..., 6,7 and related results: geometric proofs for combinatorial theorems.}  Geom. Dedicata (2013) 166:15-29.

\bibitem{Klein} Klein, F., \textit{Vorlesungen \"uber das Ikosaeder}.
Teubner, Leipzig (1884).

\bibitem{LZ}  Lando, S. K., Zvonkin, A. K., \emph{Graphs on Surfaces and Their Applications,} Springer-Verlag,
Berlin, 2004.

\bibitem{NZ} Neftin, D., Zieve, M.,
\textit{Monodromy groups of indecomposable coverings}. Preprint (2017).

\bibitem{NZ2} Neftin, D., Zieve, M.,
\textit{Monodromy groups of product type}. Preprint (2017).

\bibitem{Pakovitch} Pakovitch, F., \emph{Solution of the Hurwitz problem for Laurent polynomials.} J. Knot Theory Ramifications, 18, No. 2 (2009), 271--302.

\bibitem{PasPet}
{\sc Pascali, M.~A., Petronio, C.,}
\textit{Surface branched covers and geometric $2$-orbifolds. }
Trans. AMS. 
361, No. 11 (2009) 5885--5920.

\bibitem{Thom}
Thom, R., \emph{L'\'{e}quivalence d'une fonction diff\'{e}rentiable et d'un polynome}.
Topology 3 (1965), 297--307.

\bibitem{Voe}
V\"olklein, H.,
\emph{Groups as Galois groups. An introduction}.
Cambridge Studies in Advanced Mathematics 53, Cambridge Univ.\ Press, New York 1996.


\bibitem{Zheng} Zheng, H.,
\textit{Realizability of branched coverings of $S^2$}.
Topology and its Applications 153 (2006), 2124--2134.

\bibitem{Z} Zieve, M.,
\textit{Personal communication.} 
\end{thebibliography}
\end{document}